%% file: BosKS17_arxiv2.tex
\documentclass{article}

\usepackage{amsfonts}
\usepackage{amsmath, amsthm, amssymb}

\usepackage{enumitem}

\usepackage{url}

\usepackage{subfig}
\usepackage{booktabs,tabularx}
\usepackage{relsize}
\usepackage{pxfonts} 

  \usepackage{xcolor}
  \usepackage{tikz}
  \usetikzlibrary{plotmarks}

\usepackage{pgfplots}
\pgfplotsset{compat=1.5}

\newcommand{\hp}{s}

\newcommand{\dx}{\,dx} 

\newtheorem{theorem}{Theorem}
\newtheorem{definition}[theorem]{Definition}
\newtheorem{corollary}[theorem]{Corollary}
\newtheorem{remark}[theorem]{Remark} 
 
\begin{document}
%%%%% title : short title may not be used but TITLE is required.
% \title{TITLE}
% \title[short title]{TITLE}
\title{Preconditioning of a coupled Cahn--Hilliard Navier--Stokes system}

%%%%% author(s) :
% multiple authors:
% Note the use of \affil and \affilnum to link names and addresses.
% The author for correspondence is marked by \corrauth.
% use \emails to provide email addresses of authors
% e.g. below example has 3 authors, first author is also the corresponding
%      author, author 1 and 3 having the same address.
% \author[Zhang Z R et.~al.]{Zhengru Zhang\affil{1}\comma\corrauth,
%       Author Chan\affil{2}, and Author Zhao\affil{1}}
% \address{\affilnum{1}\ School of Mathematical Sciences,
%          Beijing Normal University,
%          Beijing 100875, P.R. China. \\
%           \affilnum{2}\ Department of Mathematics,
%           Hong Kong Baptist University, Hong Kong SAR}
% \emails{{\tt zhang@email} (Z.~Zhang), {\tt chan@email} (A.~Chan),
%          {\tt zhao@email} (A.~Zhao)}
% \footnote and \thanks are not used in the heading section.
% Another acknowlegments/support of grants, state in Acknowledgments section
% \section*{Acknowledgments}

\author{Jessica Bosch\thanks{Department of Computer Science, The University of British Columbia, Vancouver, BC, V6T 1Z4, Canada (jbosch@cs.ubc.ca).} 
\and Christian Kahle\thanks{Chair of Optimal Control, Technical University M\"unchen, Boltzmannstr.~3, 85748 Garching bei M\"unchen (Christian.kahle@ma.tum.de).} 
\and Martin Stoll\thanks{Numerical Linear Algebra for Dynamical Systems, Max Planck Institute for Dynamics of Complex Technical Systems, Sandtorstr.~1, 39106 Magdeburg, Germany (stollm@mpi-magdeburg.mpg.de).}}

%%%% maketitle %%%%%
\maketitle

%%%%% Begin Abstract %%%%%%%%%%%
\begin{abstract}
Recently, Garcke et al. [H. Garcke, M. Hinze, C. Kahle, Appl. Numer. Math. 99
(2016), 151--171)] developed a consistent discretization
scheme for a thermodynamically consistent diffuse interface model 
for incompressible two-phase flows with different densities
[H. Abels, H. Garcke, G. Gr{\"u}n, Math. Models Methods Appl. Sci. 22(3)
(2012)].
At the heart of this method lies the solution of large and sparse linear systems that arise in a 
semismooth Newton method. 

In this work we propose the use of 
preconditioned Krylov subspace solvers using effective Schur complement approximations. 
Numerical results illustrate the efficiency of our approach. In particular, our preconditioner is shown to be robust 
with respect to parameter changes.

%%%%% AMS/PACs/Keywords %%%%%%%%%%%
\noindent{\bf Keywords:} Navier--Stokes, Cahn--Hilliard, two-phase flow,
preconditioning, Schur complement approximation
\end{abstract}
%%%%% end %%%%%%%%%%%

%-----------------------------------------------------------------------------
%-----------------------------------------------------------------------------
%-----------------------------------------------------------------------------
%--------------------introduction-------------------------------
%-----------------------------------------------------------------------------
%-----------------------------------------------------------------------------

\section{Introduction}
\label{sec:intro}
%label identifier: *:intro:*

In recent years numerical simulation of two-phase flow and multiphase flow has
attained growing attention both from the modeling point of view and the
development of schemes for the numerical realization. This especially holds
for so called diffuse interface or phase-field models, where the distribution of
the two phases, i.e.
the regions where the two phases are located, is encoded in a smooth indicator
function, called phase field.
In this work we consider the particular problem of solving linear systems
arising during Newton's method for the solution of nonlinear equations 
appearing in an energy stable time stepping scheme for a 
thermodynamically consistent model for two-phase flow.

We use the thermodynamically consistent model proposed in
\cite{AbelsGarckeGruen_CHNSmodell} and the numerical scheme presented in
\cite{GarckeHinzeKahle_CHNS_AGG_linearStableTimeDisc}. The model
is based on a phase field representation of the
two-phase fluid structure and
 couples the Navier--Stokes equations to the Cahn--Hilliard equations in a
 thermodynamically consistent way, i.e., an energy equality holds.
  
Before we state the  model and the scheme below in Section~\ref{sec:CHNS},
let us briefly comment on the development of models for two-phase flows and
especially on numerical schemes for these. The first model for two-phase flow
using a phase field representation is the  model 'H' proposed in
\cite{HohenbergHalperin1977}. It couples the Navier--Stokes equation and the
Cahn--Hilliard equation \cite{CahH58} in a thermodynamically consistent way.
It is only valid for fluids of equal density. Since then
several attempts have been made for a generalization to fluids of different
densities, see e.g. \cite{Lowengrub_CahnHilliard_and_Topology_transitions,
ShenYang_CHNS_DingSpelt_Consistent,
Boyer_two_phase_different_densities}. 
The first thermodynamically consistent model, that contains the same coupling
between the Navier--Stokes and the Cahn--Hilliard model is proposed in
\cite{AbelsGarckeGruen_CHNSmodell}. Especially for equal density fluids, this
model equals the model 'H'.

In the few years since the invention of this model, several groups proposed
schemes for its numerical realization and we refer here only to
\cite{Tierra_Splitting_CHNS,
GruenMetzger__CHNS_decoupled,
Gruen_Klingbeil_CHNS_AGG_numeric,
Aland__time_integration_for_diffuse_interface,
Hintermueller_Keil_Wegner__OPT_CHNS_density,
GarckeHinzeKahle_CHNS_AGG_linearStableTimeDisc}.
The first two schemes decouple the Navier--Stokes and the Cahn--Hilliard
equation. This means that on each time instance, the Cahn--Hilliard equation can
be solved in advance and the solution can be used to solve the Navier--Stokes
equation afterwards. The systems are decoupled 
by using an augmented velocity field in the Cahn--Hillard equation,
that then only depends on the velocity field from the old time instance and thus
can be solved before turning to the Navier--Stokes equation. This procedure
adds some diffusion outside of the interface to the Cahn--Hilliard
equation. In this way, the defects from using the old velocity field can be
balanced with numerical dissipation that arises due to time discretization to
obtain an energy stable scheme.

The other papers propose fully coupled schemes
and especially in \cite{Aland__time_integration_for_diffuse_interface} benefits from using a fully coupled scheme are investigated.
At the heart of the numerical realization of such fully coupled schemes large
and sparse linear systems have to be solved, that have a block structure, namely
\begin{displaymath}
 \mathcal{A}=\left(\begin{array}{c|c}
\boldsymbol{A}_{\textup{NS}} & \boldsymbol{C}_{I}\\\hline
\boldsymbol{C}_{T} & \boldsymbol{A}_{\textup{CH}}
\end{array}\right),
\end{displaymath}
where $\boldsymbol A_{\textup{NS}}$ and $ \boldsymbol A_{\textup{CH}}$ are
themselves block matrices arising from the discretization of the (linearized) Navier--Stokes equation
($\boldsymbol{A}_{NS}$)
and the (linearized) Cahn--Hilliard equation ($\boldsymbol{A}_{CH}$),
while $\boldsymbol{C}_T$ describes coupling through transport of the
interface and $\boldsymbol{C}_{I}$ describes coupling through interfacial
forces.

The arising sparse linear systems are usually of very large dimension, and in combination with three-dimensional experiments, the application of direct solvers such as 
UMFPACK \cite{Davis16} quickly becomes infeasible. As a result, iterative methods have to be employed 
(see, e.g., \cite{Greenbaum97,Saad03} for introductions to this field). The state-of-the-art iterative schemes are the so-called Krylov subspace methods with prominent 
representatives found in \cite{cg,minres,gmres}. The convergence behavior of the iterative schemes typically depends on the conditioning of the 
problem and the clustering of the eigenvalues. 
These properties are usually affected by the model parameters. Various parameters of different scales are 
involved in $\mathcal{A}$: Spatial 
mesh sizes, the time step size, the interfacial parameter, the mobility of the interface, a penalty parameter, 
the Reynolds number, and the surface tension. 
Therefore, the convergence behavior of the iterative solver needs to be enhanced using preconditioning techniques 
$\mathcal{A}\mathcal{P}^{-1}\tilde{\boldsymbol{z}}=\boldsymbol{b}$ or 
$\mathcal{P}^{-1}\mathcal{A}\boldsymbol{z}=\mathcal{P}^{-1}\boldsymbol{b}$, where $\mathcal{P}$ is an 
invertible matrix that is easy to invert and resembles $\mathcal{A}$. Our goal in this paper is the derivation and analysis of a preconditioner well suited for the coupled system $\mathcal{A}.$ To the best of our knowledge this is the first time that the preconditioning of the coupled Navier-Stokes-Cahn-Hilliard system is studied. While preconditioning for the Navier-Stokes equations has been investigated for a long time (see \cite{ElmSW05} and the references given therein), the development of preconditioners for Cahn-Hilliard systems is a rather new topic with several contributions found in \cite{BosSB14,BoyDN11}. The coupling of both systems poses serious challenges regarding robustness of the method with respect to the crucial system parameters. 
Hence, we derive an efficient preconditioner $\mathcal{P}$ tailored to the coupled Cahn--Hilliard Navier--Stokes system using
effective Schur complement approximations and (algebraic) multigrid developed for 
elliptic systems \cite{Falgout06,Saad03,RugS87}. For this we utilize recent developments for the preconditioning of the Navier-Stokes equations using a pressure-convection-diffusion approximation to the Navier-Stokes Schur complement \cite{ElmSW05} and a multigrid for the $(1,1)$-block of the discrete Navier-Stokes system. We then approximate the Schur complement of the overall system using the recently established techniques for the Cahn-Hilliard equations that are in turn based on an efficient approximation of the Schur complement of this block using a matching argument, well established in PDE-constrained optimization \cite{PeaW12,PSW11}.

The paper is organized as follows. In Section \ref{sec:CHNS}, we introduce the
two-phase flow equation under consideration. We state the thermodynamically
consistent discretization proposed in
\cite{GarckeHinzeKahle_CHNS_AGG_linearStableTimeDisc} and review major
analytical results. In Section \ref{sec:pre}, we develop and investigate our
preconditioner, while in Section \ref{sec:num}, we show its performance in a
benchmark example and investigate its robustness with respect to  the relevant
parameters of the system under investigation. 
Finally, Section \ref{sec:concl} summarizes our findings.

%-----------------------------------------------------------------------------
%-----------------------------------------------------------------------------
%-----------------------------------------------------------------------------
%--------------------the CHNS model-------------------------------
%-----------------------------------------------------------------------------
%-----------------------------------------------------------------------------

\section{The Cahn--Hilliard Navier--Stokes model}
\label{sec:CHNS}
%label identifier *:CHNS:*

In \cite{AbelsGarckeGruen_CHNSmodell}, the following thermodynamically consistent
diffuse interface model for the simulation of two-phase fluids with different densities
is proposed:

Let $\Omega \subset \mathbb{R}^n$, $n\in\{2,3\}$ denote an open, bounded domain
with Lipschitz boundary and outer normal $\nu_\Omega$, and $I = (0,T]$ denote a
time interval.
Inside $\Omega$ there are two immiscible phases and we introduce a smooth
indicator function $\varphi$, called phase field, such that $\varphi \approx +1$
indicates the one phase and $\varphi \approx -1$ indicates the
second phase. The values $|\varphi| \leq 1$ indicate a diffuse interface between
the phases.
The velocity and the pressure of the fluid are denoted by $v$ and
$p$; see \cite{AbelsGarckeGruen_CHNSmodell} for the precise definition of this 
terms. 
Adding a chemical potential $\mu$ as first variation of the underlying
Ginzburg--Landau energy, the model is given by the following set of equations:
\begin{align}
  \rho\partial_t v + \left( \left( \rho v + J\right)\cdot\nabla
  \right)v
  - \mbox{div}\left(2\eta Dv\right) + \nabla p =& \mu\nabla \varphi + \rho g &&
  \forall x\in\Omega,\, \forall t \in I,\label{eq:CHNS:CHNSstrong1}\\
  -\mbox{div}(v) = &0&&
  \forall x\in\Omega,\, \forall t \in I,\label{eq:CHNS:CHNSstrong2}\\
  \partial_t \varphi + v \cdot\nabla \varphi - b\Delta \mu = &0&&
  \forall x\in\Omega,\, \forall t \in I,\label{eq:CHNS:CHNSstrong3}\\
  -\sigma\epsilon \Delta \varphi + \frac{\sigma}{\epsilon}W'(\varphi) - \mu =&
  0&& \forall x\in\Omega,\, \forall t \in I,\label{eq:CHNS:CHNSstrong4}\\
  v(0,x) =& v_0(x)      &&    \forall x \in
  \Omega,\label{eq:CHNS:CHNSstrongIC1}\\
  \varphi(0,x) =& \varphi_0(x) &&\forall x \in
  \Omega,\label{eq:CHNS:CHNSstrongIC2}\\
  v(t,x) =& 0 &&\forall x \in \partial \Omega,\, \forall t \in
  I,\label{eq:CHNS:CHNSstrongBC1}\\
  \nabla \mu(t,x)\cdot \nu_\Omega =
  \nabla \varphi(t,x) \cdot \nu_\Omega =& 0  &&\forall x \in \partial
  \Omega,\, \forall t \in I,\label{eq:CHNS:CHNSstrongBC2}
\end{align}
with $J  = - \frac{\rho_2-\rho_1}{2}b\nabla \mu$.

Here $\rho = \rho(\varphi) := \frac{\rho_2-\rho_1}{2}\varphi +
\frac{\rho_2+\rho_1}{2}$
describes the interpolated density of the fluid, and
$\eta = \eta(\varphi) := \frac{\eta_2-\eta_1}{2} +
\frac{\eta_2+\eta_1}{2}\varphi$
the interpolated viscosity, where $\rho_1,\rho_2$ are the densities of the two
pure phases, while $\eta_1,\eta_2$ are the corresponding viscosities.
By $2Dv = \nabla v +(\nabla v)^t$ we denote the symmetrized gradient.
The gravitational acceleration is denoted by $g$. By $b$ we denote the constant
mobility of the interface.
The scaled surface tension (see \cite[Sec. 4.3.4]{AbelsGarckeGruen_CHNSmodell})
is given by $\sigma$ and the interfacial region between the phases has a
thickness of order $\mathcal{O}(\epsilon)$.
The free energy density of the interface is denoted by $W$ and here we use
\begin{align*}
  W(\varphi) =W^\hp(\varphi)= \frac{1}{2}\left(1-\varphi^2
  + \hp \max(0,\varphi-1)^2 + \hp \min(0,\varphi+1)^2 \right),
\end{align*}
where $\hp \gg 0$ denotes a
relaxation coefficient arising from the Moreau--Yosida relaxation of the
double-obstacle free energy
\begin{align*}
  W^\infty(\varphi) =
  \begin{cases}
\frac{1}{2}(1-\varphi^2) & \mbox{ if } |\varphi|\leq 1,\\
+\infty & \mbox{else};
\end{cases}
\end{align*}
see \cite{BloweyElliott_I,HintermuellerHinzeTber}. Note that the results
proposed in \cite{GarckeHinzeKahle_CHNS_AGG_linearStableTimeDisc} are valid for
a wider class of free energy densities.
Finally, the initial velocity is given by $v_0$ and the initial phase field by
$\varphi_0$. Note that for simplicity we use a no-slip boundary condition for the
velocity field, while the scheme is literally valid as long as no flux across $\partial\Omega$
appears, i.e. $ v\cdot\nu_\Omega  \equiv 0$ on $\partial\Omega$.
This especially includes free-slip conditions
\begin{align*}
  v \cdot \nu_\Omega = 0,
  \quad \nu_\Omega^{\perp}\cdot (2\eta Dv) \nu_\Omega = 0 
  \quad \mbox{ on } \partial\Omega,
\end{align*}
where $\nu_\Omega^\perp$ denotes the tangential on $\partial\Omega$.

Concerning existence and uniqueness of solutions, we refer to
\cite{AbelsDepnerGarcke_CHNS_AGG_exSol,
AbelsDepnerGarcke_CHNS_AGG_exSol_degMob,
Gruen_convergence_stable_scheme_CHNS_AGG}.

% \begin{assumption}
% \label{ass:CHNS:rhoeta_bound}
% \hspace{2cm}
% \begin{itemize}
%   \item There exist $0 < \underline \rho \leq \overline \rho < +\infty$ such
%   that $\underline \rho \leq \rho(\varphi) \leq \overline \rho$. Due to the linear
%   nature of the function $\rho(\varphi)$, this essentially is a $L^\infty(\Omega)$
%   bound on the phase field $\varphi$.
%   \item There exist $0 < \underline \eta \leq \overline \eta < +\infty$ such
%   that $\underline \eta \leq \eta(\varphi) \leq \overline \eta$. Due to the linear
%   nature of the function $\eta(\varphi)$, this also is a
%   $L^\infty(\Omega)$ bound on the phase field $\varphi$.
% \end{itemize}
% For our choice of the free energy, in
% \cite{HintermuellerSchielaWollner_length_PD_path_MY_path_following,Kahle_Linfty_bound}
% it is shown that 
% \begin{align*}
%   \|\max(0,\varphi-1) + \min(0,\varphi+1)\|_{L^\infty(\Omega)} \leq C \hp^{-1}
% \end{align*}
% holds and thus we consider this assumption reasonable.
% Concerning other approaches to guarantee positivity of $\rho$ we refer to
% \cite{Gruen_Klingbeil_CHNS_AGG_numeric,Tierra_Splitting_CHNS,Hintermueller_Keil_Wegner__OPT_CHNS_density}.
% \end{assumption}

\bigskip

The following weak formulation of
\eqref{eq:CHNS:CHNSstrong1}--\eqref{eq:CHNS:CHNSstrongBC2} is proposed in
\cite{GarckeHinzeKahle_CHNS_AGG_linearStableTimeDisc}. 
\begin{definition}
We call $v$, $p$, $\varphi$, $\mu$ a
weak solution to \eqref{eq:CHNS:CHNSstrong1}--\eqref{eq:CHNS:CHNSstrongBC2}
if $v(0) = v_0$, $ \varphi(0) = \varphi_0$, and
\begin{align}
  \frac{1}{2}\int_{\Omega}\left( \partial_t(\rho v)+\rho\partial_t v \right) w\,dx
  +\int_{\Omega}2\eta Dv:Dw\,dx&\nonumber\\
  +a(\rho v + J,v,w) -(p,\mbox{div}w)
  =   \int_\Omega \mu\nabla \varphi w +\rho g w\,dx &\quad \forall w\in
  H^1_0(\Omega)^n, \label{eq:CHNS:CHNS1_weak}\\
  -(\mbox{div}v,q) = 0 &\quad \forall q\in
  L^2_{(0)}(\Omega), \label{eq:CHNS:CHNS2_weak}\\
  \int_\Omega\left(\partial_t\varphi + v\cdot \nabla \varphi\right) \Psi\,dx
  + \int_\Omega b\nabla \mu \cdot \nabla \Psi\,dx =0
  &\quad \forall \Psi \in  H^1(\Omega),
  \label{eq:CHNS:CHNS3_weak}\\
  \sigma \epsilon\int_\Omega\nabla \varphi\cdot\nabla \Phi\,dx
  +\frac{\sigma}{\epsilon}\int_\Omega W'(\varphi)\Phi\,dx
  - \int_\Omega \mu\Phi\,dx = 0 &\quad \forall \Phi \in H^1(\Omega),
  \label{eq:CHNS:CHNS4_weak}
\end{align}
is satisfied for almost all $t\in I$.
\end{definition}
Here, for $u\in L^3(\Omega)^n$, $v,w \in H^1(\Omega)^n$, we define the
antisymmetric trilinear form
\begin{align*}
  a(u,v,w) := \frac{1}{2} \left( \int_\Omega (( u\nabla) v) w \dx - \int_\Omega
  (( u\nabla) w) v \dx \right),
\end{align*}
and by $L^2_{(0)}(\Omega)$ we denote the space of square integrable functions
with mean value zero, i.e.
\begin{align*}
  L^2_{(0)}(\Omega) := \{ f \in L^2(\Omega)\,|\, |\Omega|^{-1}(f,1) = 0 \}.
\end{align*}
We stress, that formulation
\eqref{eq:CHNS:CHNS1_weak}--\eqref{eq:CHNS:CHNS4_weak} is based on a
reformulation of \eqref{eq:CHNS:CHNSstrong1}, that requires that $\rho$ is a
linear function of $\varphi$.
In \cite{AbelsBreit_weakSolution_nonNewtonian_DifferentDensities}  a
generalization of \eqref{eq:CHNS:CHNSstrong1}--\eqref{eq:CHNS:CHNSstrongBC2} is
presented, where this requirement might be dropped. This has to be
investigated in the future.

\begin{definition}
The energy of the system is the sum of the kinetic energy of the fluid and the
Ginzburg--Landau energy of the interface. Thus we define
\begin{align*}
  E(t) := E(v(t),\varphi(t)) := \int_\Omega\frac{1}{2} \rho |v|^2\dx +
  \sigma\int_\Omega\frac{\epsilon}{2} |\nabla \varphi |^2 
  +\frac{1}{\epsilon}W(\varphi)\dx.
\end{align*}
\end{definition}

The model \eqref{eq:CHNS:CHNSstrong1}--\eqref{eq:CHNS:CHNSstrongBC2} is called
\textit{thermodynamically consistent}, because the following energy
identity holds.

\begin{theorem}
[{\cite[Thm. 1]{GarckeHinzeKahle_CHNS_AGG_linearStableTimeDisc}}]
Assume that there exists a sufficiently smooth solution to
\eqref{eq:CHNS:CHNS1_weak}--\eqref{eq:CHNS:CHNS4_weak}.

Then, the following energy identity holds for $0<\tau<t$
\begin{equation}
\label{eq:CHNS:enerEst}
\begin{aligned}
E(t) + &\int_\tau^t \int_\Omega 2\eta(\varphi(s)) |Dv(s)|^2 + b|\nabla \mu(s)|^2
\dx\,ds \\
&= E(\tau) + \int_\tau^t \int_\Omega  \rho(\varphi(s)) g  v(s)\dx\,ds.
\end{aligned}
\end{equation}
This means that the gain and loss of energy can exactly be measured and that
in absence of outer forces (i.e. $g\equiv 0$) the energy can not increase with
time.
\end{theorem}

In \cite{GarckeHinzeKahle_CHNS_AGG_linearStableTimeDisc} a discretization scheme
is proposed, that resembles this property in the fully discrete setting.  We
summarize that scheme in the following.

%%%%%%%%%%%%%%%%%%%%%%%%%%%%%%%%%%%%%%%%%%%%%%%%%%%%%%%%%%%%%%%%%%%%%%%

\subsection{The stable scheme from \cite{GarckeHinzeKahle_CHNS_AGG_linearStableTimeDisc}}
\label{ssec:FD}
%label identifier *:FD:*
 For a numerical realization, we  discretize
\eqref{eq:CHNS:CHNS1_weak}--\eqref{eq:CHNS:CHNS4_weak} using 
a semi-implicit Euler discretization in time and the finite element method in
space.

Let $t_{-1}<0 = t_0 < t_1 < \ldots < t_{k-2}< t_{k-1} < t_k<\ldots<t_K=T$ denote
an equidistant subdivision of $I = (0,T]$ with fix step size $\tau := t_1-t_0$.
 Further, let $\mathcal T^k = \{ T_1,\ldots,T_{\Theta_k}\}$
denote a conforming
triangulation of $\overline \Omega$ with closed cells $ (T_i)_{i=1}^{\Theta_k}$,
where we assume that $\overline \Omega = \bigcup_{i=1}^{\Theta_k} T_i$. Here,
$k$ refers to the time instance, and we stress that we use different triangulations
on different time instances due to adaptive meshing.
On $\mathcal{T}^k$, we define the finite element spaces
\begin{align*} 
  V^k_1 &:= \{ v \in C(\overline\Omega) \,|\, v|_{T} \in \Pi^1(\Omega)\, \forall
  T \in \mathcal{T}^k \}
   =: \mbox{span}\{ b_1^i\}_{i=1}^{N_1^k},\\
  V^k_2 &:= \{ v \in C(\overline\Omega)^n \,|\, v|_{T} \in (\Pi^2(\Omega))^n\,
  \forall T \in \mathcal{T}^k, v|_{\partial \Omega} = 0 \}
   =: \mbox{span}\{  b_2^i\}_{i=1}^{N_2^k}.
\end{align*}
Here, $\Pi^l$ denotes the space of polynomials up to order $l$.
 We further introduce an $H^1$-stable prolongation
operator $P^k:H^1(\Omega) \to V^k_1$.
Possible operators are, e.g., the Cl\'ement operator or Lagrangian
interpolation, where we have to restrict the preimage to $C(\overline\Omega)
\cap H^1(\Omega)$.

To state the fully discrete approximation, we introduce $\varphi^k_h \in
V^k_1$  as fully discrete approximation of $\varphi(t_k)$,
$\mu^k_h \in V^k_1$ as fully discrete approximation of $\mu(t_k)$,
 $v^k_h \in V^k_2$ as fully discrete approximation of $v(t_k)$,
  and  $p_h^k\in V^k_1$ as fully discrete approximation of $p(t_k)$.
   Note that we use LBB-stable Taylor--Hood elements for the
  discretization of the Navier--Stokes equation.
  We further use the abbreviations $\rho^k := \rho(\varphi^k_h)$, $\eta^k :=
  \eta(\varphi^k_h)$.

\bigskip

The fully discrete variant of
\eqref{eq:CHNS:CHNS1_weak}--\eqref{eq:CHNS:CHNS4_weak}
is  as follows.

Let $\varphi^{k-2} \in V^{k-2}_1$, 
$\varphi^{k-1} \in V^{k-1}_1$,
$\mu^{k-1}\in V^{k-1}_1$, 
$v^{k-1} \in V^{k-1}_2$ be given.
Find $\varphi^k \in V^k_1$,
 $\mu^k \in V^k_1$, 
 $v^k\in V^k_2$, 
 and $p_h^k\in
V^k_1$ such that
\begin{align}
  \frac{1}{\tau}\left( \frac{\rho^{k-1} + \rho^{k-2} }{2}v_h^k -
  \rho^{k-2}v^{k-1},w\right)
  + a(\rho^{k-1}v^{k-1}+J^{k-1},v_h^{k},w)\nonumber\\
  +(2\eta^{k-1}Dv^{k}_h,D w) 
  -(p_h^k,\mbox{div} w)
  -(\mu^{k}_h\nabla\varphi^{k-1}+ \rho^{k-1}  g,w)&= 0 \,
  \forall w \in  V^k_2,\label{eq:FD:chns1}\\
 -(\mbox{div} v_h^k,q) &=0 \,
  \forall q \in  V^k_1,\label{eq:FD:chns2}\\
  \frac{1}{\tau}(\varphi^{k}_h-P^{k}\varphi^{k-1},\Psi)
  +(b\nabla   \mu^{k}_h,\nabla \Psi) -(v^{k}_h\varphi^{k-1},\nabla \Psi)
  &=0 \, \forall \Psi \in V^k_1,\label{eq:FD:chns3}\\
  \sigma\epsilon(\nabla \varphi^{k}_h,\nabla   \Phi)
  +\frac{\sigma}{\epsilon}\left(W_+^\prime(\varphi^{k}_h)+W^\prime_-(P^{k}\varphi^{k-1}),\Phi\right)
  -(\mu^{k}_h,\Phi)
  &=0 \, \forall \Phi \in V^k_1. \label{eq:FD:chns4}
\end{align}

By $W_+(\varphi) = \frac{\hp}{2}(\max(0,\varphi-1)^2 +
\min(0,\varphi+1)^2)$
we denote the convex part of $W$, and by $W_-(\varphi) =
\frac{1}{2}(1-\varphi^2)$ we denote the concave part.

\begin{remark}
Note that  \eqref{eq:FD:chns1} -- \eqref{eq:FD:chns4}
is a two-step scheme since it requires data from $t_{k-2}$ and $t_{k-1}$ to
evaluate the solution at time $t_k$. Especially, we require the somewhat
artificial data $\varphi^{-1}$.
In \cite{GarckeHinzeKahle_CHNS_AGG_linearStableTimeDisc}, 
a one-step scheme with a different time discretization is proposed for the
initialization of the two-step scheme  \eqref{eq:FD:chns1} --
\eqref{eq:FD:chns4}. Here, we argue as in
\cite{Hintermueller_Keil_Wegner__OPT_CHNS_density} that given $\varphi^{-1}$
and $v^0$ one can solve
\eqref{eq:FD:chns3} -- \eqref{eq:FD:chns4} to obtain the
missing data $\varphi^0$ and $\mu^0$ to start the two-step scheme.
Compare also \cite[Rem. 3]{GarckeHinzeKahle_CHNS_AGG_linearStableTimeDisc}
\end{remark}

\begin{theorem}
[{\cite[Thm. 2, Thm. 3]{GarckeHinzeKahle_CHNS_AGG_linearStableTimeDisc}}]
\label{thm:FD:exSol}
Let $\varphi^{k-2}\in V^{k-2}_1$, 
$\varphi^{k-1} \in V^{k-1}_1$, 
$\mu^{k-1} \in V^{k-1}_1$, 
$v^{k-1} \in V^{k-1}_2$ be given.
Then there exists a unique solution $v_h^k \in V^k_2$, $\varphi_h^k, \mu_h^k,
p^k_h \in V^k_1$ to \eqref{eq:FD:chns1}--\eqref{eq:FD:chns4} that
fulfills the following energy inequality
\begin{equation}
  \label{eq:FD:energyInequality}
  \begin{aligned}
    \int_\Omega \frac{1}{2}\rho^{k-1}\left|v^{k}_h\right|^2\dx
    + \sigma \int_\Omega \frac{\epsilon}{2}|\nabla \varphi^{k}_h|^2
    +\frac{1}{\epsilon}W(\varphi^{k}_h) \dx\\
    + \frac{1}{2}\int_\Omega \rho^{k-2}|v^{k}_h-v^{k-1}|^2\dx
    + \frac{\sigma\epsilon}{2}
    \int_\Omega |\nabla \varphi^{k}_h-\nabla P^k \varphi^{k-1}|^2\dx\\
    + \tau\int_\Omega 2\eta^{k-1}|Dv^{k}_h|^2\dx
    +\tau\int_\Omega b|\nabla \mu^{k}_h|^2\dx \\
    \leq
    \int_\Omega \frac{1}{2}\rho^{k-2}\left|v^{k-1}\right|^2\dx
    + \sigma\int_\Omega\frac{\epsilon}{2} |\nabla P^{k}\varphi^{k-1}|^2
    +\frac{1}{\epsilon}W(P^{k}\varphi^{k-1})\dx\\
    +\tau \int_\Omega  \rho^{k-1} g v^{k}_h\dx.
  \end{aligned}
\end{equation}
Moreover, this solution can be found by Newton's method.
\end{theorem}
\begin{proof}
The existence of a solution is shown using Brouwer's fixpoint theorem. The
energy inequality then follows from using 
$w = v^k_h$, $q = p_h^k$, $\Psi= \mu_h^k$ and 
$\Phi =\frac{1}{\tau}(\varphi_h^k - P^k\varphi^{k-1})$, summing the equations up and
using the properties of $W_+$ and $W_-$.
Then the uniqueness follows from considering two different solutions and showing
that they are equal, using the energy inequality.

To show that Newton's method is applicable, we have to show that system
\eqref{eq:FD:chns1}--\eqref{eq:FD:chns4} is Newton-differentiable and that 
the derivative is continuously invertible. Beside $W_+$ all terms are Frech\'et
differentiable and thus Newton differentiable. $W_+$ is Newton differentiable as
shown in \cite{Hintermueller_primaldualSet}. The derivative has a structure,
that is very similar to \eqref{eq:FD:chns1}--\eqref{eq:FD:chns4} and the
continuous invertability follows similar to  the existence of a unique
solution.
\end{proof}

\begin{remark}
Note that the energy inequality \eqref{eq:FD:energyInequality} bounds the energy
at time instance $k$ by the energy of $P^k\varphi^{k-1}$, i.e., the prolongation
of $\varphi^{k-1}$ to the triangulation $\mathcal T^k$, thus not by the 
energy at the old time instance.
To overcome this, in \cite{GarckeHinzeKahle_CHNS_AGG_linearStableTimeDisc}, a
postprocessing step is added to the adaptive concept that guarantees that the energy does not
increase through prolongation, where Lagrangian interpolation is used as
prolongation $P^k$. Including this step, we guarantee that in absence of outer
forces the energy of the system can not increase and
\eqref{eq:FD:energyInequality} resembles a fully discrete counterpart of
\eqref{eq:CHNS:enerEst}.
\end{remark}

\begin{remark}
Using the unique solution proposed in Theorem~\ref{thm:FD:exSol} in
\cite{GarckeHinzeKahle_CHNS_AGG_linearStableTimeDisc} the existence of a unique
solution to a corresponding time-discrete variant of
\eqref{eq:FD:chns1}--\eqref{eq:FD:chns4} is shown.
\end{remark}

Theorem \ref{thm:FD:exSol} states that we find the unique solution to
\eqref{eq:FD:chns1}--\eqref{eq:FD:chns4} 
by Newton's method.
To state the algorithm, let us compactly write  
\eqref{eq:FD:chns1}--\eqref{eq:FD:chns4} 
as 
\begin{align*}
  \left<F(v_h^k,p_h^k,\varphi_h^k,\mu_h^k),(w,q,\Psi,\Phi) \right> :=
  \left(
  (F^1(\ldots),w),
  (F^2(\ldots),q),
  (F^3(\ldots),\Psi),
  (F^4(\ldots),\Phi)
  \right)^t,
\end{align*}
where $(F^1(\ldots),w)$ abbreviates \eqref{eq:FD:chns1},
$(F^2(\ldots),\Psi)$ abbreviates \eqref{eq:FD:chns2},
$(F^3(\ldots),\Phi)$ abbreviates \eqref{eq:FD:chns3}, and
$(F^4(\ldots),\Psi)$ abbreviates \eqref{eq:FD:chns4}.
Then Newton's method generates the following sequence
\begin{align}
  DF(x^m)\delta x &= -F(x^m), \label{eq:FD:newtEq}\\
  x^{m+1} &= x^m + \delta x,\quad m=1,\ldots\nonumber
\end{align}
where $x$ abbreviates $(v_h^k,p_h^k,\varphi_h^k,\mu_h^k)$. 
We note that $F$ is due to our choice of $W$ only Newton differentiable, see
\cite{Hintermueller_primaldualSet,HintermuellerHinzeTber}.
A Newton derivative of $F$ is given by
\begin{equation}
  \label{eq:FD:DF}
\begin{aligned}
  \left< DF(v_h^k,p_h^k,\varphi_h^k,\mu_h^k)\right.&\left.(\delta v,\delta p,
  \delta \varphi,\delta \mu),(w,q,\Psi,\Phi)\right>:= \\
  &\frac{1}{\tau}\left( \frac{\rho^{k-1} + \rho^{k-2} }{2}\delta v,w\right)
  + a(\rho^{k-1}v^{k-1}+J^{k-1},\delta v,w)\\
  &+(2\eta^{k-1}D\delta v,D w)
  -(\delta p,\mbox{div}w)
  -(\delta \mu \nabla\varphi^{k-1},w)\\
  &-(\mbox{div} \delta v, q)\\
  &+\frac{1}{\tau}(\delta\varphi,\Psi)
  +(b\nabla   \delta \mu,\nabla \Psi)
  -(\delta v  \varphi^{k-1},\nabla\Psi)\\
  &+\sigma\epsilon(\nabla \delta \varphi,\nabla   \Phi)
  +\frac{\sigma}{\epsilon}
  (W_+^{\prime\prime}(\varphi^{k}_h)\delta  \varphi,\Phi)
  -(\delta\mu,\Phi),
\end{aligned}
\end{equation}
where
\begin{align*}
  W_+^{\prime\prime}(\varphi_h^k(x)) :=
  \begin{cases}
  		\hp & \mbox{if } |\varphi_h^k(x)| > 1,\\
  		0 & \mbox{else}
  \end{cases}
  \quad
  \forall x \in \Omega.
\end{align*}
Solving \eqref{eq:FD:newtEq} leads to the numerical solution of a large linear
system that we pose next.

Using the basis $\{b_2^i\} \subset V^k_2$ and $\{b_1^i\} \subset V^l_1$ we 
define the following matrices and vectors:  
  \begin{align*}
    & & & A := M_2 + T_a + K_2, & & & \\
    & M_2 = (m^2_{ij})_{i,j=1}^{N_2},&
    & T_a = (t^a_{ij})_{i,j=1}^{N_2},&
    & K_2 = (k^2_{ij})_{i,j=1}^{N_2},&\\
    & m^2_{ij}:=
  \left(\frac{\rho^{k-1}+\rho^{k-2}}{2\tau}b_2^j,b_2^i\right),&
  & t^a_{ij} := a(\rho^{k-1}v^{k-1} + J^{k-1},b_2^j,b_2^i), &
  &k^2_{ij}:=   (2\eta^{k-1}Db_2^j,Db_2^i),
  \end{align*}
  \begin{align*}
   & B := (b_{ij})_{i=1,\ldots,N_1}^{j=1,\ldots,N_2},&
   &  U := (\xi_{ij})_{i=1,\ldots,N_2}^{j=1,\ldots,N_1},&
  &   T := (t_{ij})_{i=1,\ldots,N_1}^{j=1,\ldots,N_2},\\
   &  b_{ij} = -(\mbox{div}b_2^j,b_1^i), &
   &  \xi_{ij} = -(b_1^j\nabla \varphi^{k-1},b_2^i), &
   & t_{ij} = (b_2^j \varphi^{k-1},\nabla b_1^i),
  \end{align*}
  \begin{align*}
     &M_1 := (m^1_{ij})_{ij = 1,\ldots,N_1}, &
     & K_1 := (k^1_{ij})_{ij = 1,\ldots,N_1}, &
      &\Lambda := (\lambda_{ij})_{ij=1}^{N_1},\\
      & m^1_{ij} = (b_1^j,b_1^i), &&
       k^1_{ij} = (\nabla b_1^j,\nabla b_1^i),&&
       \lambda_{ij} = (W_+^{\prime\prime}(\varphi^m)b_1^j,b_1^i),
  \end{align*}
  \begin{align*}
  F_1 = (f^1_j)_{j=1}^{N_2},& \quad 
  f^1_j = (F^1(x^m),b_2^j), &
  F_2 = (f^2_j)_{j=1}^{N_1},& \quad 
  f^2_j = (F^2(x^m),b_1^j),\\
  F_3 = (f^3_j)_{j=1}^{N_1},& \quad
  f^3_j = (F^3(x^m),b_1^j), &
  F_4 = (f^4_j)_{j=1}^{N_1},& \quad 
  f^4_j = (F^4(x^m),b_1^j).
\end{align*} 
Here, $\varphi^m$ denotes the $m-th$ iterate of Newton's method for the
approximation of $\varphi_h^k$.
 Equation \eqref{eq:FD:newtEq} can than be
written as
% \begin{equation}
% \label{eq:FD:LS}
%   \left(\begin{array}{cc|cc}
%   A &B^t & 0                   & U   \\
%   B & 0  & 0                   & 0   \\\hline
%   T & 0          & \tau^{-1}M_1 & bK_1\\
%   0 & 0   & \sigma \epsilon K_1+\sigma\epsilon^{-1}\Lambda & -M_1
%   \end{array}\right)
%   \begin{pmatrix}
%   \boldsymbol{\delta v}\\
%   \boldsymbol{\delta p}\\
%   \boldsymbol{\delta \varphi}\\
%   \boldsymbol{\delta \mu}
%   \end{pmatrix}
%   =
%   \begin{pmatrix}
%   F_1\\
%   F_2\\
%   F_3\\
%   F_4
%   \end{pmatrix},
% \end{equation}
\begin{equation}
 \label{eq:FD:LS}
\left(\begin{array}{cc|cc}
  A &B^t & U                   & 0   \\
  B & 0  & 0                   & 0   \\\hline
  0 & 0  & M_1								 & -\sigma \epsilon K_1-\sigma\epsilon^{-1}\Lambda\\  
  \tau T & 0  & \tau bK_1					 & M_1
  \end{array}\right)
  \begin{pmatrix}
  \boldsymbol{\delta v}\\
  \boldsymbol{\delta p}\\
  \boldsymbol{\delta \mu}\\
  \boldsymbol{\delta \varphi}
  \end{pmatrix}=
  \begin{pmatrix}
  F_1\\
  F_2\\
  -F_4\\
  F_3
  \end{pmatrix}.
\end{equation}
where $\boldsymbol{\delta v},
\boldsymbol{\delta p}, 
\boldsymbol{\delta \mu},  
\boldsymbol{\delta \varphi}$
denote the node vectors for $\delta v, \delta p, \delta \varphi, \delta \mu$.  Note that here we
already did some reformulation in order to maintain our later analysis.
In the following, we denote the coefficient matrix in \eqref{eq:FD:LS} by $\mathcal{A}$.
We can further write $\mathcal{A}$ as
\begin{equation}
\label{eq:FD:LS_short}
	\mathcal{A}=\left(\begin{array}{c|c}
\boldsymbol{A}_{\textup{NS}} & \tilde{\boldsymbol{C}}_{I}\\\hline
\tilde{\boldsymbol{C}}_{T} & \tilde{\boldsymbol{A}}_{\textup{CH}}
\end{array}\right),
\end{equation}
where the blocks $\boldsymbol{A}_{\textup{NS}}$ and $\tilde{\boldsymbol{A}}_{\textup{CH}}$ are the 
discrete realizations of the  Navier--Stokes and linearized Cahn--Hilliard system, respectively. 
Their coupling is represented by $\tilde{\boldsymbol{C}}_{I}$, the coupling through the interfacial 
force, and $\tilde{\boldsymbol{C}}_{T}$, the coupling through the transport at the interface. 
The matrix $A$ is invertible by Lax-Milgram's theorem (note the antisymmetric transport term), and 
the $(1,1)$ block $\boldsymbol{A}_{\textup{NS}}$  
then is invertible since we use LBB-stable elements. 
The $(2,2)$ block $\tilde{\boldsymbol{A}}_{\textup{CH}}$ is invertible; see, e.g., \cite{HintermuellerHinzeTber}. 
The Schur complement of the whole system, 
$\boldsymbol{S}=\tilde{\boldsymbol{A}}_{\textup{CH}}-\tilde{\boldsymbol{C}}_{T}\boldsymbol{A}_{\textup{NS}}^{-1}\tilde{\boldsymbol{C}}_{I}$, 
describes a Cahn--Hilliard system with additional transport. 
$\boldsymbol{S}$ is at least invertible for small time steps; see, e.g., \cite{Kahle14}. 
% i.e., as long as the perturbation with a singular matrix is not too large.

%-----------------------------------------------------------------------------
%-----------------------------------------------------------------------------
%-----------------------------------------------------------------------------
%--------------------the preconditioner---------------------------------------
%-----------------------------------------------------------------------------
%-----------------------------------------------------------------------------

\section{Preconditioning}
\label{sec:pre}
%label identifier: *:pre:*
As we have seen in the previous section, a large and sparse linear nonsymmetric system is at the heart of the computation. 
In \cite{GarckeHinzeKahle_CHNS_AGG_linearStableTimeDisc}, the system $(\ref{eq:FD:LS})$ is solved by preconditioned GMRES \cite{SaaS86}
 with a restart after 10 iterations. The authors use the block diagonal preconditioner
\begin{equation}
\label{eq:FD:LS:P_GHK}
\mathcal{P}=\left(\begin{array}{cc} \boldsymbol{\hat{A}}_{\textup{NS}} & \boldsymbol{0}\\ \boldsymbol{0} & \tilde{\boldsymbol{A}}_{\textup{CH}}\end{array}\right).
\end{equation}
The (2,2) block $\tilde{\boldsymbol{A}}_{\textup{CH}}$ is inverted by LU decomposition. The (1,1) block $\boldsymbol{\hat{A}}_{\textup{NS}}$ is an upper block triangular preconditioner of the form
\begin{equation}
\label{chap3:CHNS:NS_prec}
\boldsymbol{\hat{A}}_{\textup{NS}}=\left(\begin{array}{cc} \hat{A} & B^{t}\\ 0 & \hat{S}_{\textup{NS}}\end{array}\right).
\end{equation}
$\hat{A}$ is composed of the diagonal blocks of $A$ and is inverted by LU decomposition. $\hat{S}_{\textup{NS}}$ is an approximation of the exact Schur complement $S_{\textup{NS}}=-BA^{-1}B^{t}$ of the Navier--Stokes system. Garcke et al.~\cite{GarckeHinzeKahle_CHNS_AGG_linearStableTimeDisc} use
\begin{equation}
\label{chap3:CHNS:NS_Schur_prec}
	\hat{S}_{\textup{NS}}=-K_{p}A_{p}^{-1}M_{p},
\end{equation}
where 
$M_{p}=M_1$ is the pressure mass matrix, 
$K_{p}=K_1$ the pressure Laplacian matrix, and 
$A_{p}$ is the representation of $A$ on the pressure space, i.e., 
$A_{p} := M_{2,p} + T_{a,p} + K_{2,p}$, where
\begin{align*}
  M_{2,p} = (m^2_{ij})_{i,j=1}^{N_1},&\quad m^2_{ij}:=
  \left(\frac{\rho^{k-1}+\rho^{k-2}}{2\tau}b_1^j,b_1^i\right),\\
  T_{a,p} = (t^a_{ij})_{i,j=1}^{N_1},& \quad 
  t^a_{ij} := a(\rho^{k-1}v^{k-1} + J^{k-1},b_1^j,b_1^i),\\
  K_{2,p} = (k^2_{ij})_{i,j=1}^{N_1},&\quad k^2_{ij}:=
  (2\eta^{k-1}Db_1^j,Db_1^i).
\end{align*}
This Schur complement approximation was proposed, e.g., in~\cite{KayLW02,ElmSW05}, where it was shown to be independent of 
the mesh size and only mildly dependent on the Reynolds number.\\

Our aim here is to further improve the performance of the preconditioner as a tailored iterative method for the system 
$(\ref{eq:FD:LS})$ can reduce the computing time dramatically. In order to achieve this, we first derive a general strategy for the 
coupled system and then show that for this to work well both the Navier-Stokes part of the discretized equation as well as the Cahn-Hilliard part need to
be approximated by sophisticated strategies. 
% Now, the current paper concerns the fully iterative solution of the system $(\ref{eq:FD:LS})$. 
% \todo[inline]{TODO: As it seems, more emphasis and clarity should be employed at this point, 
% since this sentence seems to convey the main novelty of the paper.}
This is based on the preconditioning techniques that have been developed in \cite{BosSB14} for 
the nonsmooth Cahn--Hilliard system together with the methods that have been developed 
in~\cite{KayLW02,ElmSW05} for the Navier--Stokes equations and where partly previously used in \cite{GarckeHinzeKahle_CHNS_AGG_linearStableTimeDisc}. 
Let us restate the linear system \eqref{eq:FD:LS} for convenience here
\begin{equation}
 \label{eq:FD:LS2}
\left(\begin{array}{cc|cc}
  A &B^t & U                   & 0   \\
  B & 0  & 0                   & 0   \\\hline
  0 & 0  & M_1								 & -\sigma \epsilon K_1-\sigma\epsilon^{-1}\Lambda\\  
  \tau T & 0  & \tau bK_1					 & M_1
  \end{array}\right)
  \begin{pmatrix}
  \boldsymbol{\delta v}\\
  \boldsymbol{\delta p}\\
  \boldsymbol{\delta \mu}\\
  \boldsymbol{\delta \varphi}
  \end{pmatrix}=
  \begin{pmatrix}
  F_1\\
  F_2\\
  -F_4\\
  F_3
  \end{pmatrix},
\end{equation}
and we recall that we denote the coefficient matrix in \eqref{eq:FD:LS2} by $\mathcal{A}$ that we
can partition as
\begin{equation}
\label{eq:FD:LS2_short}
	\mathcal{A}=\left(\begin{array}{c|c}
\boldsymbol{A}_{\textup{NS}} & {\boldsymbol{C}}_{I}\\\hline
{\boldsymbol{C}}_{T} & \boldsymbol{A}_{\textup{CH}}
\end{array}\right).
\end{equation}
Our basis for preconditioning is the upper block triangular preconditioner
\begin{displaymath}
\mathcal{P}=\left(\begin{array}{cc} \boldsymbol{A}_{\textup{NS}} & \boldsymbol{C}_{I}\\ \boldsymbol{0} & \boldsymbol{S}\end{array}\right),
\end{displaymath}
motivated in \cite{ElmSW05,MurGW00}, where $\boldsymbol{S}=\boldsymbol{A}_{\textup{CH}}-\boldsymbol{C}_{T}\boldsymbol{A}_{\textup{NS}}^{-1}\boldsymbol{C}_{I}$ is the Schur complement of the whole system 
as introduced above after Equation (\ref{eq:FD:LS_short}). 
The right-preconditioned matrix becomes
\begin{displaymath}
 \mathcal{A}\mathcal{P}^{-1}=\left(\begin{array}{c|c}
\boldsymbol{I} & \boldsymbol{0}\\\hline
\boldsymbol{C}_{T}\boldsymbol{A}_{\textup{NS}}^{-1} & \boldsymbol{I}
\end{array}\right),
\end{displaymath}
where $\boldsymbol{I}$ is the identity matrix. Hence, $\mathcal{A}\mathcal{P}^{-1}$ has only a 
single eigenvalue of $1$, and $\mathcal{P}$ is called a theoretical optimal preconditioner.\footnote{Note, that 
the left-preconditioned system $\mathcal{P}^{-1}\mathcal{A}$ has the same spectrum as the right-preconditioned 
system $\mathcal{A}\mathcal{P}^{-1}$.} This preconditioner needs the application of the inverse of $\boldsymbol{A}_{\textup{NS}}$ 
and of $\boldsymbol{S}$, which cannot be explicitly used. Hence, our aim is the development of practical 
approximations $\boldsymbol{\hat{A}}_{\textup{NS}}\approx\boldsymbol{A}_{\textup{NS}}$ and 
$\boldsymbol{\hat{S}}\approx\boldsymbol{S}$, which leads to our practical preconditioner
\begin{equation}
\label{eq:FD:LS:P}
\mathcal{P}_{\textup{out}}=\left(\begin{array}{cc} \boldsymbol{\hat{A}}_{\textup{NS}} & \boldsymbol{C}_{I}\\ \boldsymbol{0} & \boldsymbol{\hat{S}}\end{array}\right).
\end{equation}

Note that the index of $\mathcal{P}_{\textup{out}}$ marks an outer preconditioner. 
Below in Section \ref{ssec:pre_Schur}, we will introduce an inner preconditioner for the solution of linear systems 
with $\boldsymbol{\hat{S}}$. 
For good approximations $\boldsymbol{\hat{A}}_{\textup{NS}}\approx\boldsymbol{A}_{\textup{NS}}$ and 
$\boldsymbol{\hat{S}}\approx\boldsymbol{S}$, the preconditioned matrix $\mathcal{A}\mathcal{P}_{\textup{out}}^{-1}$ has 
only a small number of different eigenvalue clusters. This in turn is known to result in only a 
few iterations of suitable Krylov subspace solvers until convergence \cite{ElmSW05,MurGW00}.

\subsection{Approximation of $\boldsymbol{A}_{\textup{NS}}$}
\label{ssec:pre_NS}
%label identifier *:pre_NS:*
Similar to Garcke et al.~\cite{GarckeHinzeKahle_CHNS_AGG_linearStableTimeDisc} in (\ref{chap3:CHNS:NS_prec}), we choose $\boldsymbol{\hat{A}}_{\textup{NS}}$ as 
\begin{equation}
\label{eq:FD:LS:P_NS}
\boldsymbol{\hat{A}}_{\textup{NS}}=\left(\begin{array}{cc} \hat{A} & B^{t}\\ 0 & -\hat{S}_{\textup{NS}}\end{array}\right).
\end{equation}
As above, $\hat{A}$ is composed of the diagonal blocks of $A$. 
We use an algebraic multigrid (AMG) preconditioner\footnote{AMG methods typically exhibit geometric-like properties for positive definite elliptic type operators but use only algebraic information. This has the advantage that AMG can work well even for complicated geometries and meshes. We refer to \cite{RugS87,Falgout06} for more information on AMG. We also want to emphasize that geometric multigrid (see, e.g., \cite{Wesseling92,Hackbusch85}) approximations are also well suited to approximate $\hat{A}$ provided they can be readily applied.} for the approximation of the inverse of $\hat{A}$. 
As shown in \cite{Ramage99}, 
multigrid is a good preconditioner for convection-diffusion problems if the Reynolds number 
is not too large. 
$\hat{S}_{\textup{NS}}$ is given in (\ref{chap3:CHNS:NS_Schur_prec}). 
The action of the inverse of $M_{p}$ and $K_{p}$ are performed with an AMG each. 
Note that the conjugate gradient method with Jacobi preconditioning provides a good approximation to the 
inverse of $M_{p}$ as well, see e.g.~\cite{KayLW02}. 
$K_{p}$ is a discrete Laplacian for which multigrid provides a good approximation to the inverse; see \cite{Wesseling92}.\\

\subsection{Approximation of $\boldsymbol{S}$}
\label{ssec:pre_Schur}
%label identifier *:pre_Schur:*
Now, let us consider the Schur complement $\boldsymbol{S}=\boldsymbol{A}_{\textup{CH}}-\boldsymbol{C}_{T}\boldsymbol{A}_{\textup{NS}}^{-1}\boldsymbol{C}_{I}$ of the whole system. 
Using the preconditioner $\boldsymbol{\hat{A}}_{\textup{NS}}$ from the previous section, we approximate $\boldsymbol{S}$ as
\begin{displaymath}
\boldsymbol{S}\approx\boldsymbol{A}_{\textup{CH}}-\boldsymbol{C}_{T}\boldsymbol{\hat{A}}_{\textup{NS}}^{-1}\boldsymbol{C}_{I}
=\left(\begin{array}{cc}
M_{1} & -\sigma\epsilon K_{1}-\sigma\epsilon^{-1}\Lambda\\
 \tau b K_{1} -\tau T\hat{A}^{-1}U & M_{1}
\end{array}\right)=:\boldsymbol{\hat{S}}.
\end{displaymath}
%We propose to approximate the action of the inverse of $\boldsymbol{S}$ by a preconditioned GMRES iteration applied to the system of the form $\boldsymbol{\hat{S}}\boldsymbol{y}=\boldsymbol{f}$. 
We propose to apply a preconditioned GMRES iteration to the system of the form $\boldsymbol{\hat{S}}\boldsymbol{y}=\boldsymbol{f}$. 
We call this iteration the inner iteration. 
For the construction of the preconditioner $\mathcal{P}_{\textup{in}}$ for the inner iteration, we make use of the following formulation:
\begin{displaymath}
\boldsymbol{\hat{S}}
=\left(\begin{array}{cc}
M_{1} & -\sigma\epsilon K_{1}-\sigma\epsilon^{-1}\Lambda\\
 \tau b K_{1} & M_{1}
\end{array}\right)-\left(\begin{array}{cc}
0 & 0\\
\tau T\hat{A}^{-1}U & 0
\end{array}\right)=\boldsymbol{A}_{\textup{CH}}-\left(\begin{array}{cc}
0 & 0\\
\tau T\hat{A}^{-1}U & 0
\end{array}\right)\approx \boldsymbol{A}_{\textup{CH}}.
\end{displaymath}

Hence, we build the preconditioner $\mathcal{P}_{\textup{in}}$ for the inner iteration on the basis of the simplification 
$\boldsymbol{A}_{\textup{CH}}$ of $\boldsymbol{\hat{S}}$.\\

In the following, 
we study the importance of taking the whole block $\boldsymbol{A}_{\textup{CH}}$ 
as preconditioner $\mathcal{P}_{\textup{in}}$. Remember from Equation (\ref{eq:FD:LS2})
\begin{displaymath}
  \boldsymbol{A}_{\textup{CH}}=\left(\begin{array}{cc}
   M_1 & -\sigma \epsilon K_1-\sigma\epsilon^{-1}\Lambda\\  
   \tau bK_1 & M_1
  \end{array}\right).
\end{displaymath}
Now, if we would further simplify  $\boldsymbol{A}_{\textup{CH}}$ to
\begin{displaymath}
\boldsymbol{\mathring{A}}_{\textup{CH}}
=\left(\begin{array}{cc}
M_{1} & -\sigma\epsilon K_{1}\\
 \tau b K_{1} & M_{1}
\end{array}\right),
\end{displaymath}
we could easily derive an optimal preconditioner for $\boldsymbol{\mathring{A}}_{\textup{CH}}$; 
see, e.g.,\cite{AxeN11,BoyDN11}. 
However, Theorem~\ref{thm:spectrum} and Corollary~\ref{cor:spectrum} below 
state the severe influence of the penalty parameter $s$ contained in $\Lambda$; 
see also \cite{Bosch16}. 
Hence, a simplification to $\boldsymbol{\mathring{A}}_{\textup{CH}}$ would give a worse 
approximation for large values of $s$. This justifies our approach of taking 
$\boldsymbol{A}_{\textup{CH}}$ 
as preconditioner $\mathcal{P}_{\textup{in}}$.

For the following theorem, we make use of the symmetric positive definiteness 
of $M_{1}$ as well as of the symmetric positive 
semidefiniteness of $K_{1}$ and $\Lambda$.
\begin{theorem}
\label{thm:spectrum}
With the simplified notation $M:=M_1$, $K:=\tau b K_1$, 
$\alpha=\frac{\sigma\varepsilon}{\tau b}$, $\beta=\frac{\sigma}{\varepsilon}$, let
\begin{displaymath}
\boldsymbol{X}
=\left(\begin{array}{cc}
M & -\alpha K\\
 K & M
\end{array}\right),\quad
   \boldsymbol{Y}=\left(\begin{array}{cc}
   M & -\alpha K-\beta\Lambda\\  
   K & M
  \end{array}\right).
\end{displaymath}
Note that $\boldsymbol{X}=\boldsymbol{\mathring{A}}_{\textup{CH}}$ and 
$\boldsymbol{Y}=\boldsymbol{A}_{\textup{CH}}$. 
Then, 
\begin{displaymath}
 \text{sp}(\boldsymbol{X}^{-1}\boldsymbol{Y})\subseteq B_{\varsigma}(1),
\end{displaymath}
where $\text{sp}(\cdot)$ denotes the spectrum of a matrix, and 
$B_{\varsigma}(1)$ is a circle in the complex plane around one with radius $\varsigma$. 
The radius is bounded by $\frac{\beta}{2\sqrt{\alpha}}\,\rho(\tilde{\Lambda})$, 
where $\tilde{\Lambda}=M^{-\frac{1}{2}}\Lambda M^{-\frac{1}{2}}$. 
%In particular, $N_{1}$ eigenvalues are equal to one.
\end{theorem}
\begin{proof}
It holds $\text{sp}(\boldsymbol{X}^{-1}\boldsymbol{Y})\subseteq B_{\varsigma}(1)$ if and only if 
$\text{sp}(\boldsymbol{I}-\boldsymbol{X}^{-1}\boldsymbol{Y})\subseteq B_{\varsigma}(0)$ 
if and only if 
$\rho(\boldsymbol{X}^{-1}(\boldsymbol{X}-\boldsymbol{Y}))<\varsigma$. Here, $\rho(\cdot)$ 
denotes the spectral radius of a matrix, and $\boldsymbol{I}$ is the 
identity matrix of appropriate size. 

Let us start with finding an equivalent formulation for 
$\rho(\boldsymbol{X}^{-1}(\boldsymbol{X}-\boldsymbol{Y}))$.
With the notation $C=KM^{-1}$ and  
$W=I+\alpha C^{2}$, where $I$ is again the 
identity matrix of appropriate size, we have
\begin{align*}
 \boldsymbol{X}&=\left(\begin{array}{cc}I & -\alpha C\\C & I\end{array}\right)
 \left(\begin{array}{cc}M & 0\\0 & M\end{array}\right),\\
  \boldsymbol{X}^{-1}&= \left(\begin{array}{cc}M^{-1} & 0\\0 & M^{-1}\end{array}\right)
  \left(\begin{array}{cc}I & -\alpha C\\C & I\end{array}\right)^{-1}.
\end{align*}
Using block matrix inversion, we have
\begin{displaymath}
 \left(\begin{array}{cc}I & -\alpha C\\C & I\end{array}\right)^{-1}=\ 
 \left(\begin{array}{cc}(I+\alpha C^{2})^{-1} & \alpha(I+\alpha C^{2})^{-1}C\\
       -(I+\alpha C^{2})^{-1}C & (I+\alpha C^{2})^{-1}
\end{array}\right)
=\left(\begin{array}{cc}W^{-1} & \alpha W^{-1}C\\-W^{-1}C & W^{-1}\end{array}\right)%\\
%&=\ \left(\begin{array}{cc}W^{-1} & 0\\0 & W^{-1}\end{array}\right)
% \left(\begin{array}{cc}I & \alpha C\\-C & I\end{array}\right).
\end{displaymath}
Hence,
\begin{align*}
  \boldsymbol{X}^{-1}(\boldsymbol{X}-\boldsymbol{Y})&=
   \left(\begin{array}{cc}M^{-1} & 0\\0 & M^{-1}\end{array}\right)
   \left(\begin{array}{cc}W^{-1} & \alpha W^{-1}C\\-W^{-1}C & W^{-1}\end{array}\right)
  \left(\begin{array}{cc}0 & \beta\Lambda\\0 & 0\end{array}\right)=
  \beta\,\left(\begin{array}{cc}0 & (WM)^{-1}\Lambda\\0 & -(WM)^{-1}C\Lambda\end{array}\right)
\end{align*}
and therefore
\begin{displaymath}
 \rho\left(\boldsymbol{X}^{-1}(\boldsymbol{X}-\boldsymbol{Y})\right)=\beta\,\rho\left((WM)^{-1}C\Lambda\right).
\end{displaymath}

For a further reformulation we introduce  the rational function $r(x)=\frac{x}{1+\alpha x^{2}}$, and
define $\tilde{C}=M^{-\frac{1}{2}}KM^{-\frac{1}{2}}$ and 
$\tilde{\Lambda}=M^{-\frac{1}{2}}\Lambda M^{-\frac{1}{2}}$, which are 
symmetric positive semidefinite.
Then we have
\begin{align*}
 r(\tilde{C})&=\left(I+\alpha\tilde{C}^{2}\right)^{-1}\tilde{C}=
 \left(I+\alpha M^{-\frac{1}{2}}KM^{-1}KM^{-\frac{1}{2}}\right)^{-1}M^{-\frac{1}{2}}KM^{-\frac{1}{2}}\\
 &=\left[M^{\frac{1}{2}}\left(I+\alpha M^{-\frac{1}{2}}KM^{-1}KM^{-\frac{1}{2}}\right)\right]^{-1}KM^{-\frac{1}{2}}
=\left(M^{\frac{1}{2}}+\alpha KM^{-1}KM^{-\frac{1}{2}}\right)^{-1}KM^{-\frac{1}{2}}\\
&=\left[\left(I+\alpha KM^{-1}KM^{-1}\right)M^{\frac{1}{2}}\right]^{-1}KM^{-\frac{1}{2}}
=M^{-\frac{1}{2}}\left(I+\alpha C^{2}\right)^{-1}KM^{-\frac{1}{2}}\\
&=M^{-\frac{1}{2}}W^{-1}KM^{-\frac{1}{2}}
 \end{align*}
and hence
\begin{align*}
 (WM)^{-1}C\Lambda &=M^{-1}W^{-1}KM^{-1}\Lambda=M^{-\frac{1}{2}}\left(M^{-\frac{1}{2}}W^{-1}KM^{-\frac{1}{2}}\right)\left(M^{-\frac{1}{2}}\Lambda M^{-\frac{1}{2}}\right)M^{\frac{1}{2}}\\
 &=M^{-\frac{1}{2}}r(\tilde{C})\tilde{\Lambda}M^{\frac{1}{2}}.
\end{align*}
It follows
\begin{align}
 \rho\left((WM)^{-1}C\Lambda\right)&=\rho\left(r(\tilde{C})\tilde{\Lambda}\right)\leq\|r(\tilde{C})\tilde{\Lambda}\|\leq\|r(\tilde{C})\|\;\|\tilde{\Lambda}\|=\rho\left(r(\tilde{C})\right)\,\rho\left(\tilde{\Lambda}\right)
 \leq \max_{\lambda\in\text{sp}(\tilde{C})}\left\{r(\lambda)\right\}\,\rho\left(\tilde{\Lambda}\right)\notag\\
% =r\left(\rho(\tilde{C})\right)\,\rho\left(\tilde{\Lambda}\right)\\
 &\leq \max_{x}\left\{r(x)\right\}\,\rho\left(\tilde{\Lambda}\right)=\frac{1}{2\sqrt{\alpha}}\,\rho\left(\tilde{\Lambda}\right),\label{eq1:spectrum}
\end{align}
where we have taken into account that the eigenvalues of 
$r(Z)$ for a matrix $Z$ with eigenvalues $\lambda_1,\ldots,\lambda_q$ 
are given by $r(\lambda_1),\ldots, r(\lambda_q)$.
\end{proof}

Applied to our situation, we have
 \begin{displaymath}
  \text{sp}(\boldsymbol{\mathring{A}}_{\textup{CH}}^{-1}\boldsymbol{A}_{\textup{CH}})\subseteq B_{\varsigma}(1),
 \end{displaymath}
where $\varsigma\leq\frac{\sqrt{\tau\sigma b}}{2\epsilon\sqrt{\epsilon}}\,\rho(\tilde{\Lambda})$.

\begin{corollary}
\label{cor:spectrum}
As a result from Theorem \ref{thm:spectrum}, 
we get $\varsigma\leq 0.5$ when 
 $\tau\leq\varepsilon^{3}/(s^{2}\sigma b\,\rho(\tilde{\Lambda}_{0})^{2})$, 
 where $\tilde{\Lambda}_{0}=s^{-1}M^{-\frac{1}{2}}\Lambda M^{-\frac{1}{2}}$. 
 In particular, in the case of lumped mass matrices, we have $\rho(\tilde{\Lambda}_{0})=1$ and hence, the circle radius is bounded by $\varsigma\leq\frac{s\sqrt{\tau\sigma b}}{2\epsilon\sqrt{\epsilon}}$. 
 Hence, we get $\varsigma\leq 0.5$ when $\tau\leq\varepsilon^{3}/(s^{2}\sigma b)$.
\end{corollary}
\begin{proof}
In order to illustrate the influence of the penalty parameter $\hp$, we reformulate $\tilde{\Lambda}$ as
\begin{displaymath}
 \tilde{\Lambda}=M^{-\frac{1}{2}}\Lambda M^{-\frac{1}{2}}=sM^{-\frac{1}{2}}\Lambda_{0} M^{-\frac{1}{2}}=s\tilde{\Lambda}_{0},
\end{displaymath}
where $\Lambda_{0}=(\lambda^{0}_{ij})_{ij=1}^{N_1}$, 
$\lambda^{0}_{ij} = (W_{+,0}^{\prime\prime}(\varphi^m)b_1^j,b_1^i)$, and
\begin{align*}
  W_{+,0}^{\prime\prime}(\varphi_h^k(x)) :=
  \begin{cases}
  		1 & \mbox{if } |\varphi_h^k(x)| > 1,\\
  		0 & \mbox{else}
  \end{cases}
  \quad
  \forall x \in \Omega.
\end{align*}
Hence, we have in (\ref{eq1:spectrum})
\begin{equation}
 \label{eq:FD:LS:S_GEP:10}
 \rho\left((WM)^{-1}C\Lambda\right)\leq\frac{s}{2\sqrt{\alpha}}\,\rho(\tilde{\Lambda}_{0})=\frac{s\sqrt{\tau\sigma b}}{2\epsilon\sqrt{\epsilon}}\,\rho(\tilde{\Lambda}_{0}),
\end{equation}
where $\rho(\tilde{\Lambda}_{0})$ depends on the spatial mesh size and $\varphi_h^k$. 
Therefore, for $\tau\leq\varepsilon^{3}/(s^{2}\sigma b\rho(\tilde{\Lambda}_{0})^{2})$, it holds 
$\text{sp}(\boldsymbol{\mathring{A}}_{\textup{CH}}^{-1}\boldsymbol{A}_{\textup{CH}})\subseteq B_{0.5}(1)$.\\

In the case of lumped mass matrices, $\tilde{\Lambda}_{0}$ becomes a diagonal matrix 
with entries that are either zero or one. Thus, we obtain in (\ref{eq:FD:LS:S_GEP:10})
\begin{displaymath}
\rho\left((WM)^{-1}C\Lambda\right)\leq\frac{s}{2\sqrt{\alpha}}=\frac{s\sqrt{\tau\sigma b}}{2\epsilon\sqrt{\epsilon}}.
\end{displaymath}
Therefore, for $\tau\leq \varepsilon^{3}/(s^{2}\sigma b)$, it holds 
$\text{sp}(\boldsymbol{\mathring{A}}_{\textup{CH}}^{-1}\boldsymbol{A}_{\textup{CH}})\subseteq B_{0.5}(1)$.\\
\end{proof}

Hence, neglecting the block $\Lambda$ in $\boldsymbol{A}_{\textup{CH}}$ would only be satisfying for tiny time step sizes $\tau$, which is far away from being practical.\\

As preconditioner for the inner iteration, we propose the upper block triangular preconditioner
\begin{equation}
\label{eq:FD:LS:P_CH}
\mathcal{P}_{\textup{in}}=\left[\begin{array}{cc} M_{1} & -\sigma\varepsilon K_{1}-\sigma\varepsilon^{-1}\Lambda\\ 0 & -\hat{S}_{CH}\end{array}\right].
\end{equation}
$\hat{S}_{CH}$ is an approximation of the exact Schur complement
\begin{displaymath}
S_{CH}=M_{1}+\tau\left(b K_{1} -T\hat{A}^{-1}U\right)M_{1}^{-1}\left(\sigma\varepsilon K_{1}+\sigma\varepsilon^{-1}\Lambda\right)	
\end{displaymath}
of $\boldsymbol{\hat{S}}$. 
Our procedure for the approximation of the Schur complement $S_{CH}$ originates in the work of Pearson and 
Wathen~\cite{PeaW12}, who developed preconditioners for PDE-constrained optimization. 
Their matching strategy, applied to our problem, is the following: Construct a preconditioner of the form 
$\hat{S}_{CH}=S_{1}M_{1}^{-1}S_{2}$, which captures the exact Schur complement $S_{CH}$ as close as possible. 
Note that we need $\hat{S}_{CH}$ to be nonsingular. We design $\hat{S}_{CH}$ as
\begin{align}
	\hat{S}_{CH}&=S_{1}M_{1}^{-1}S_{2}\notag\\
	&=\left(M_{1}+\sqrt{\tau\sigma b}K_{1}\right)M_{1}^{-1}\left(M_{1}+\sqrt{\tau b\sigma^{-1}}\left[\sigma\varepsilon K_{1}+\sigma\varepsilon^{-1}\Lambda\right]\right)\notag\\	
	&=M_{1}+\tau b K_{1}M_{1}^{-1}\left(\sigma\varepsilon K_{1}+\sigma\varepsilon^{-1}\Lambda\right)+\sqrt{\tau b\sigma}K_{1}+\sqrt{\tau b\sigma}\left(\varepsilon K_{1}+\varepsilon^{-1}\Lambda\right).\label{eq:FD:LS:P_CH_Schur}
\end{align}

The first term in (\ref{eq:FD:LS:P_CH_Schur}) matches the first term in the exact Schur complement. The second term in (\ref{eq:FD:LS:P_CH_Schur}) approximates the second term in the exact Schur complement. 
Due to the factor $\sqrt{\tau b\sigma}$, the influence of both remainder terms in (\ref{eq:FD:LS:P_CH_Schur}) is reduced. 
We refer the reader to \cite{BosSB14,BosKSW14,BosS15a,Bosch16} for Schur complement approximations to other Cahn--Hilliard problems. 
The action of the inverse of $S_{1}$ and $S_{2}$ is performed with an AMG each since both form 
the discretization of an elliptic operator. 
We also apply AMG for the action of the inverse of the $(1,1)$ block $M_{1}$ in $\mathcal{P}_{\textup{in}}$.
\footnote{For consistent mass matrices, the Chebyshev-iteration \cite{GolV61a,GolV61b} provides a powerful preconditioner \cite{WatR09,ReeS10}.}

Here, we finish the theoretical discussion about the preconditioner. 
In the next section, we illustrate its efficiency via various numerical experiments.

%-----------------------------------------------------------------------------
%-----------------------------------------------------------------------------
%-----------------------------------------------------------------------------
%--------------------the numerical examples-----------------------------------
%-----------------------------------------------------------------------------
%-----------------------------------------------------------------------------

\section{Numerical examples}
\label{sec:num}
%label identifier: *:num:*
In this section, we show numerical results for the presented coupled Cahn--Hilliard Navier--Stokes problem. 
First, we explain our implementation framework.\\

Garcke et al.~\cite{GarckeHinzeKahle_CHNS_AGG_linearStableTimeDisc} have implemented the whole numerical 
simulation in C++. We use their code, but the iterative solution of the linear
system is executed in \textsc{MATLAB}\textsuperscript{\textregistered} R2012a on a 32-bit server with CPU type 
Intel\textsuperscript{\textregistered} Core\textsuperscript{\texttrademark} E6850 @3.00 GHz with 2 CPUs. 
We use the MATLAB Engine API in order to call \textsc{MATLAB}\textsuperscript{\textregistered} 
from C++.\\

As we use GMRES as the inner iteration, we apply FGMRES as the outer iteration. 
FGMRES is a variant of GMRES and was introduced by Saad~\cite{Saad93}. 
This method allows changes in the preconditioner at every step. 
Note again that we apply right preconditioning here. % because of the following reason: 
%A left preconditioner modifies the right-hand side, whereas a right preconditioner does not modify it. 
%As stated in~\cite[p.~4]{CheMZ16}, a major hurdle for developing variable preconditioners for left preconditioning 
%is the disconnection between the preconditioned residuals and the actual residuals. 
We use FGMRES (outer solver) with a restart after $30$ iterations and set the initial guess to the zero vector. 
As stopping criterion, we use
\begin{displaymath}
 \|\boldsymbol{b}-\mathcal{A}\boldsymbol{z}^{(l)}\|\leq\min(10^{-6}\,\|\boldsymbol{b}\|,10^{-6}),
\end{displaymath}
%set the tolerance to
%$\min(10^{-6}\,\|\boldsymbol{b}\|^{-1},10^{-6})$, \todo{ $\| b\|^{-1}$ ??} 
where $\boldsymbol{b}$ 
denotes the corresponding right-hand side and $\boldsymbol{z}^{(l)}$ the calculated solution at FGMRES step $l$. 
We use GMRES (inner solver) without restart and set the initial guess to the zero vector, the tolerance for the preconditioned relative residual to $10^{-1}$, 
and the maximum number of iterations to $50$.\\

For the application of AMG, we employ the HSL (formerly the Harwell Subroutine Library) Mathematical Software Library, 
a collection of Fortran codes for large-scale scientific computation; see \url{http://www.hsl.rl.ac.uk/} and \cite{BoyMS10}. In particular, 
we make use of the \textit{HSL\_MI20} package for the approximation of the inverse of $\hat{A}$ 
(the two diagonal blocks of the 
discrete convection-diffusion operator) and of $K_{p}$ (the pressure Laplacian matrix). 
For $\hat{A}$, we apply \textit{HSL\_MI20} with two coarse levels in the multigrid structure, 
two V-cycles with symmetric Gauss--Seidel smoothing (two pre- and two post-smoothing iterations), and 
10 Gauss-Seidel iterations on the coarsest level. 
For $K_{p}$, we apply \textit{HSL\_MI20} with at most 100 coarse levels in the multigrid structure, 
one V-cycle with symmetric Gauss--Seidel smoothing (two pre- and two post-smoothing iterations), and 
the sparse direct solver \textit{HSL\_MA48} on the coarsest level. 

Moreover, we employ the software package \textit{AGMG} version 3.2.0, a program
which implements AMG described in \cite{Notay10} with further improvements from
\cite{NapN12} and \cite{Notay12}; see also \url{http://homepages.ulb.ac.be/~ynotay/AGMG}.
It introduces K-cycle multigrid meaning that the iterative solution of the residual equation at each level 
is accelerated with a Krylov subspace method. We use the default, which is a preconditioned variant of the generalized 
conjugate residual method (GCR) \cite{EisES83} restarted every 10 iterations. We set the maximum number of iterations to 50. 
The coarsening is stopped when the coarse grid matrix has 200 or less rows, allowing fast direct inversion with 
LAPACK routines \cite{Andal99}. Symmetric Gauss-Seidel smoothing is used with one pre- and one post-smoothing iteration. 
We make use of the \textit{AGMG} MATLAB interface for the approximation of the inverse of $M_{p}$ 
(the pressure mass matrix), 
of $M_{1}$ (the (1,1) block of the inner preconditioner $\mathcal{P}_{\textup{in}}$), as well as of $S_{1}$ and 
$S_{2}$ (the two blocks forming the Schur complement approximation $\hat{S}_{CH}$). 
If not mentioned otherwise, we set the tolerance on the relative residual norm to $10^{-3}$, $10^{-2}$, $10^{-5}$, and $10^{-5}$ 
for $M_{p}$, $M_{1}$, $S_{1}$, and $S_{2}$, respectively.%\\

%Now, we are ready for numerical results.

\subsection{Parameter study: A rising bubble}
In this section, we demonstrate the robustness of our proposed preconditioner regarding relevant model parameters. 
As test example, we use a quantitative benchmark for rising bubble dynamics; 
see the first benchmark test case in \cite{Hysing_Turek_quantitative_benchmark_computations_of_two_dimensional_bubble_dynamics}. 
A simulation is illustrated in Figure~\ref{fig:Benchmark_1}.
\begin{figure}[!htbp]
  \centering
  \begin{tabular}{ccccccc}
    \includegraphics[trim=  7cm 4cm 17cm 5cm, clip,height=3cm]{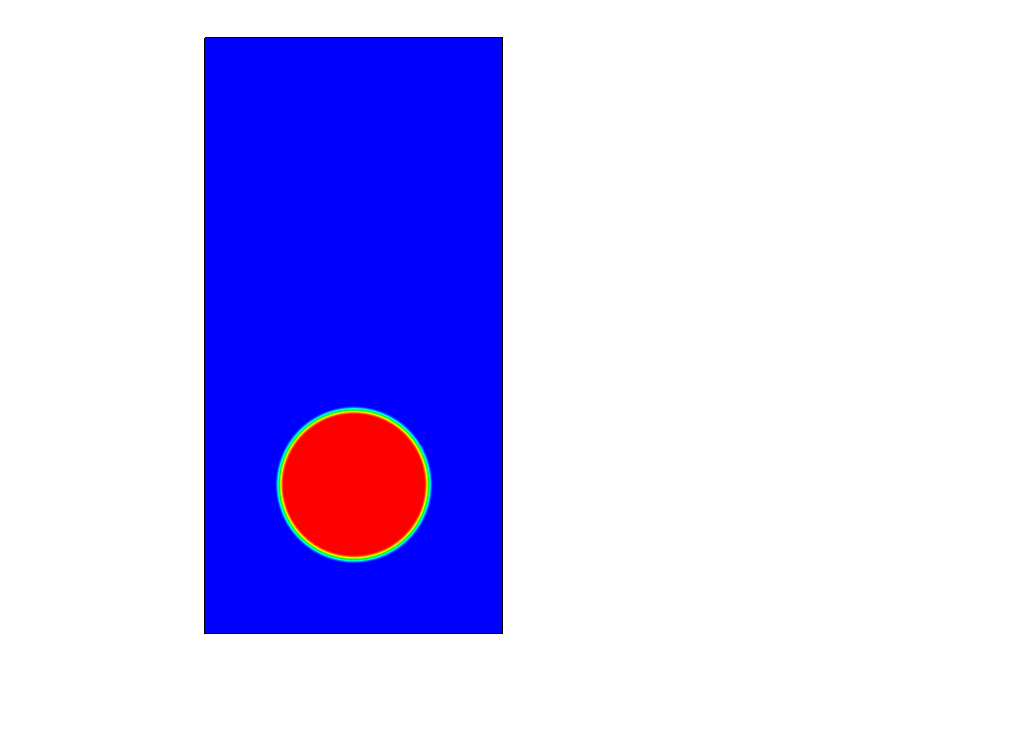} & &
    \includegraphics[trim=  7cm 4cm 17cm 5cm, clip,height=3cm]{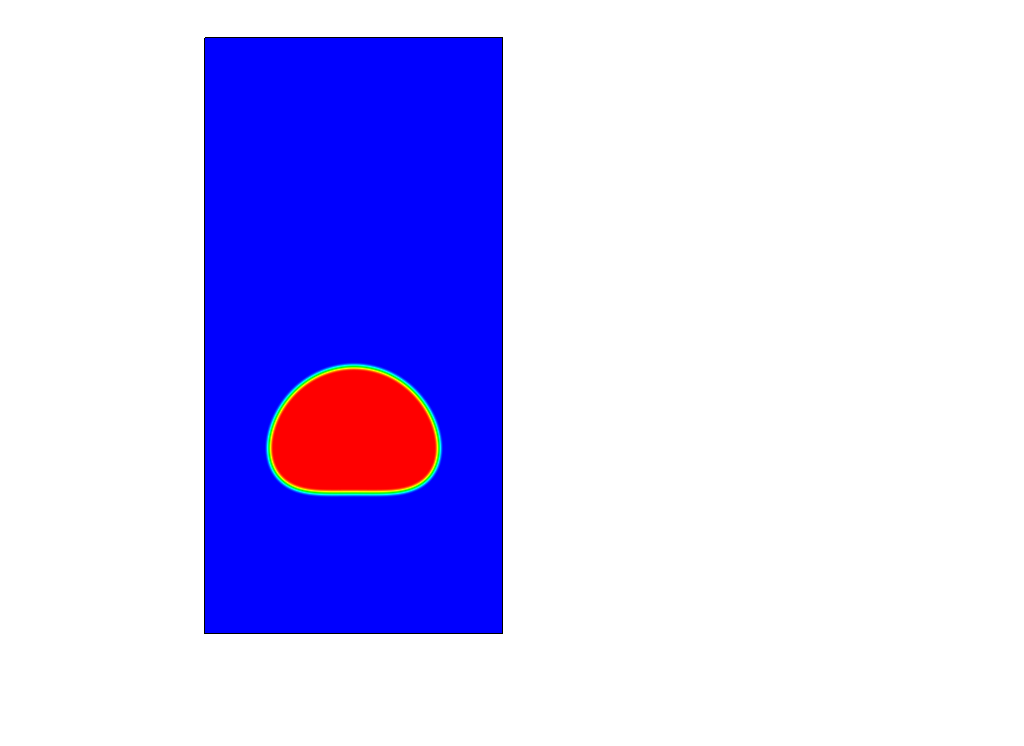} & &
    \includegraphics[trim=  7cm 4cm 17cm 5cm, clip,height=3cm]{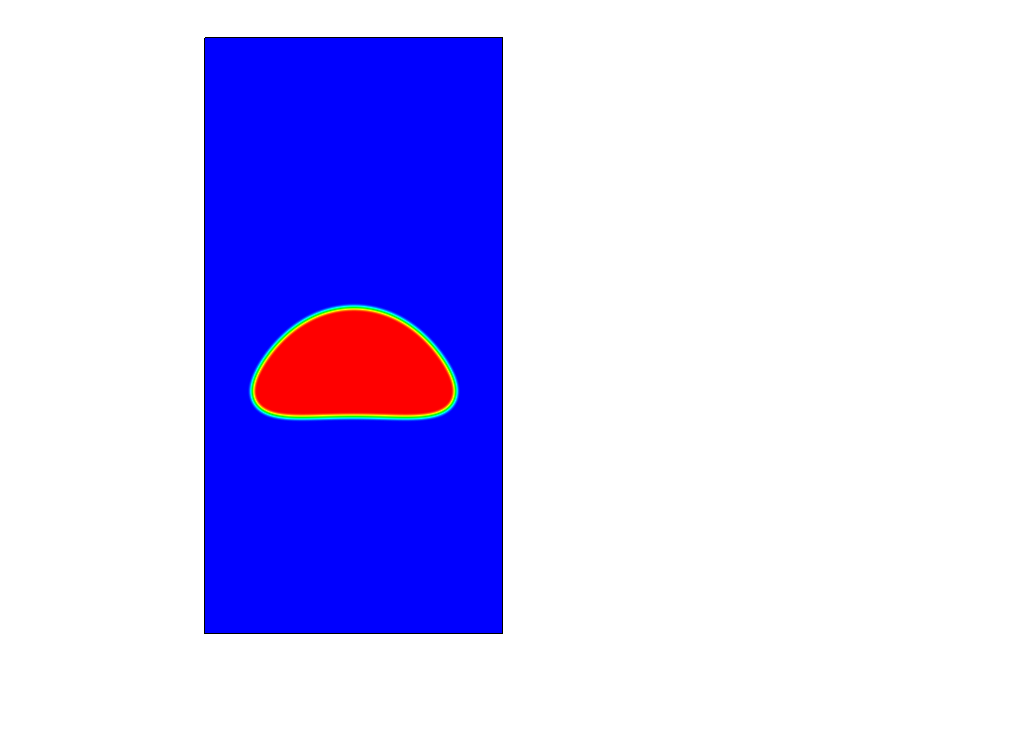} & &
    \includegraphics[trim=  7cm 4cm 17cm 5cm, clip,height=3cm]{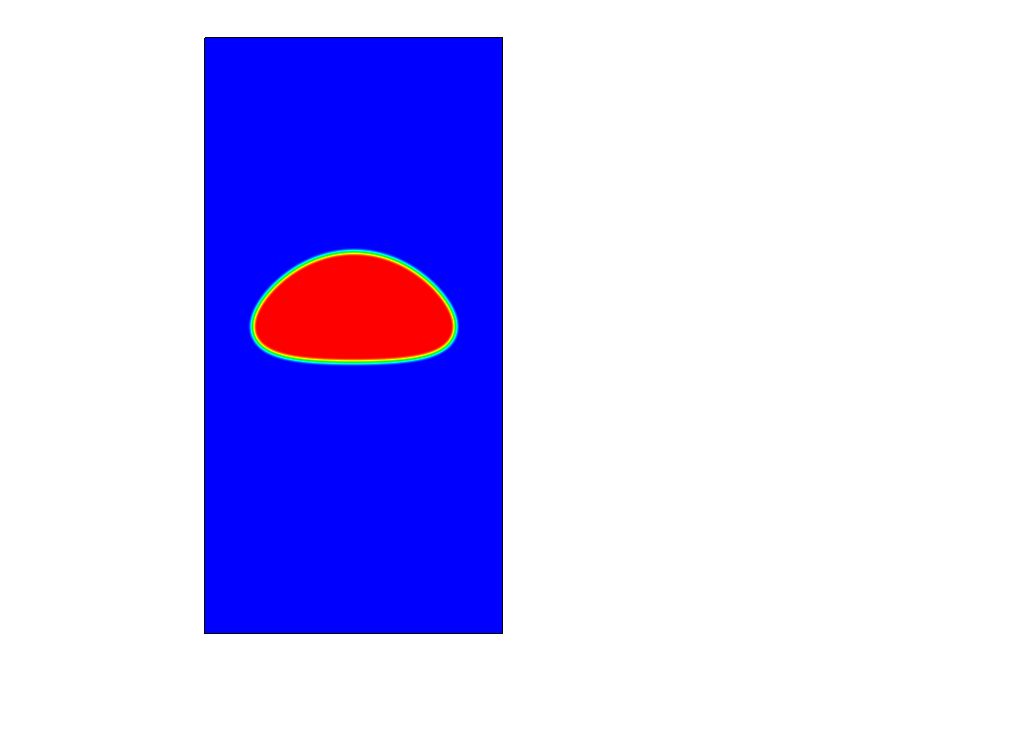}\\
    $t=0$ & & $t=1$ & & $t=2$ & & $t=3$
  \end{tabular}
  \caption{Simulation of a rising bubble using a coupled Cahn--Hilliard Navier--Stokes model.}
  \label{fig:Benchmark_1}
\end{figure}

The initial configuration is described as follows; see also~\cite[p.~168]{GarckeHinzeKahle_CHNS_AGG_linearStableTimeDisc}. 
The spatial domain is $\Omega=(0,1)\times(0,2)$ with no-slip boundary conditions for the velocity field on the top 
and bottom wall and free-slip boundary conditions on the left and right wall. The initial state consists of a bubble of 
radius $r=0.25$ centered at the spatial point $(0.5,0.5)$. The initial velocity is zero. In the following, we denote 
by $N_{1}^{0}$ and $N_{2}^{0}$ the values of $N_{1}$ and $N_{2}$ at time $0$. 
Note that during the simulation, the values of $N_{1}$ and $N_{2}$ stay about their initial sizes, i.e., 
there is no drastic change. 
The fixed parameters in all 
tests are given as $\rho_{2}=100$, $\eta_{1}=10$, $\eta_{2}=1$, $g=(0,-0.98)^{t}$. 
The remaining parameters are given below for each individual experiment. 
Moreover, the Reynolds number is given as
\begin{displaymath}
  Re=\frac{0.35\,\rho_{1}}{\eta_{1}}.
\end{displaymath} 
Figures \ref{fig:parameter} and \ref{fig:parameter_2} demonstrate the robustness with respect to different model parameters. 
Table~\ref{table:parameter_Newton} shows the values of all parameters. 
In Figure \ref{fig:parameter}\subref{subfig:parameter_1}, we simultaneously vary the mesh sizes via refinements of the 
initial spatial mesh $\mathcal{T}^{0}$, the interfacial parameter $\epsilon$, 
the time step size $\tau$, as well as the mobility $b$. In fact, this is the practical procedure: Choose an $\epsilon$ and 
adjust the mesh sizes. The time step size is adapted to fulfill the CFL-condition
\begin{displaymath}
 \max_{T}{\left\{\frac{\tau\,|v^k|_T}{\text{diam}(T)}\right\}}\leq 1.
\end{displaymath}
Moreover, the mobility is chosen to be 
$b=10^{-3}\epsilon$, as used in \cite[p.~168]{GarckeHinzeKahle_CHNS_AGG_linearStableTimeDisc}. 
Except for the initial time frame, we observe a
robust behavior of iteration numbers. For the first time step the initial data for Newton's method,
which is a discrete approximation of $\varphi_0$, seems not to be appropriate and thus resultign a
slow convergence.
 In Figure \ref{fig:parameter}\subref{subfig:parameter_2}, we vary the scaled surface tension $\sigma$. 
Although the iteration numbers behave quite chaotic, they mostly stay in the range between $50$ and $65$. 
In Figure \ref{fig:parameter}\subref{subfig:parameter_3}, we vary the Reynolds number via increasing the density 
$\rho_{1}$. We observe a small increase of iterations numbers as the Reynolds number increases. 
In Figure \ref{fig:parameter_2}\subref{subfig:parameter_1}, we vary the mobility $b$. 
We observe a benign increase of iterations numbers as the mobility decreases. 
In Figure \ref{fig:parameter_2}\subref{subfig:parameter_2}, we vary the penalty parameter $\hp$. 
Comparing the iteration numbers for $\hp=10^{4}$ with the ones for $\hp\in\{10^{6},10^{8},10^{9}\}$, we even obtain better results for the larger penalty parameters. 
Finally, Table~\ref{table:parameter_Newton} illustrates the maximum and average number of semismooth Newton iterations 
for each of the five subplots, respectively.\\
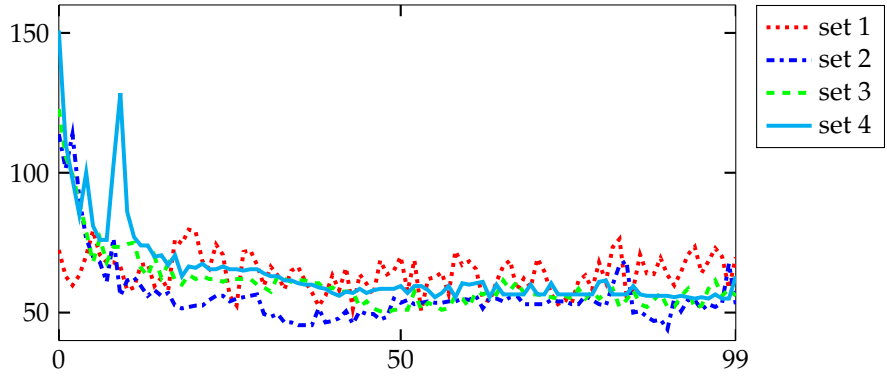
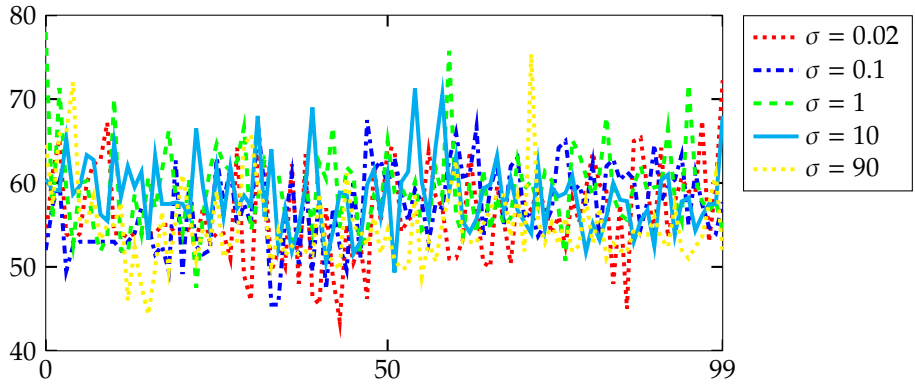
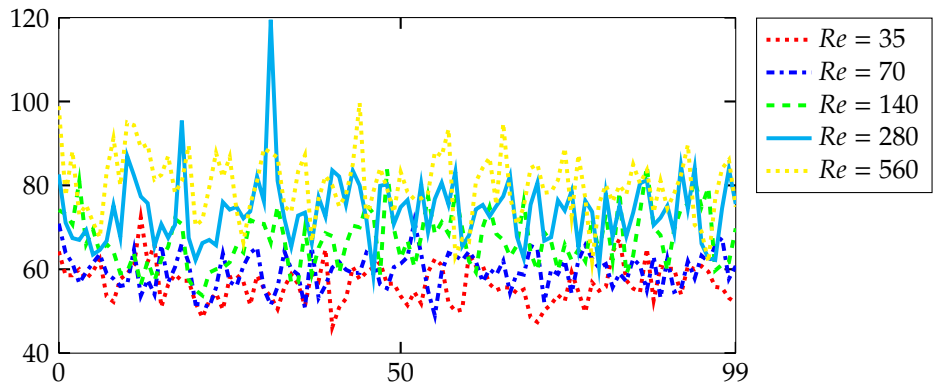
\begin{figure}[!htbp]
	\centering
	\newlength\figureheight 
	\newlength\figurewidth 
	\setlength\figureheight{4.7cm} 
	\setlength\figurewidth{6cm}
	\subfloat[Simultaneous variation of the initial spatial mesh $\mathcal{T}^{0}$, the interfacial parameter $\epsilon$, 
the time step size $\tau$, as well as the mobility $b$.]{\label{subfig:parameter_1}{\input{vary_all.tikz}}}\\
	%$\ $
	  \subfloat[Variation of the scaled surface tension $\sigma$.]{\label{subfig:parameter_2}{\input{vary_sigma.tikz}}}\\

\subfloat[Variation of the Reynolds number via increasing the density 
$\rho_{1}$.]{\label{subfig:parameter_3}{\input{vary_Re.tikz}}}
	%$\ $
%	  \subfloat[$h_{j}=2^{-j-6}$, $\varepsilon_{j}=9\, h_{j}/(2\sqrt{2}\cdot\textup{atanh}(0.95))$, $\tau_{1}=10^{-6}$, $\tau_{2}=10^{-7}$, $\tau_{3}=1.25\cdot 10^{-8}$ for $j=1,2,3$.]{\label{subfig:parameter_4}{\input{Pics_and_Data/Chapter_3/iter_smooth_semi-impl_all.tikz}}}
	\caption{Iteration numbers for the parameter study. The x-axis shows the time step and the y-axis the average number of FGMRES iterations per semismooth Newton step.}
	\label{fig:parameter}
\end{figure}

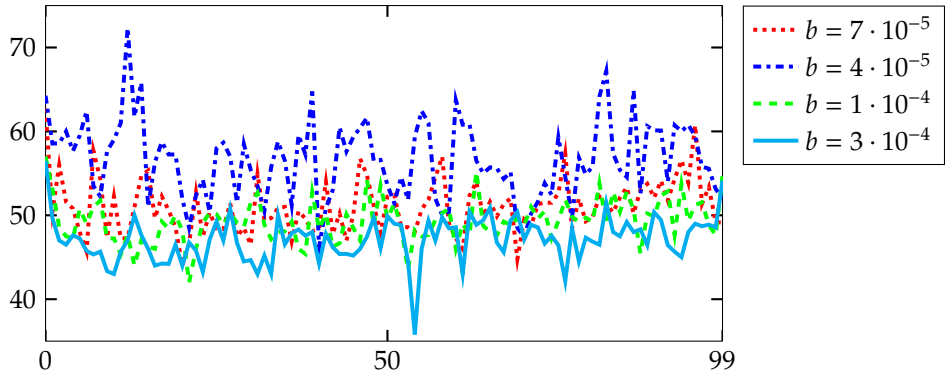
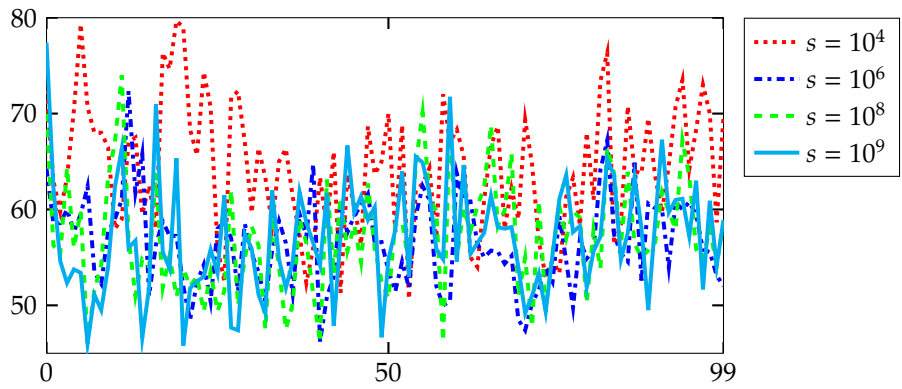
\begin{figure}[!htbp]
	\centering
	\setlength\figureheight{4.7cm} 
	\setlength\figurewidth{6cm}
	\subfloat[Variation of the mobility $b$.]{\label{subfig:parameter_4}{\input{vary_mobility.tikz}}}\\
	%$\ $
	  \subfloat[Variation of the penalty parameter $\hp$.]{\label{subfig:parameter_5}{\input{vary_penalty.tikz}}}\\

%\subfloat[$\sigma=15.5972$, $\hp=10^{6}$, $N_{1}^{0}=6599$, $N_{2}^{0}=26213$, $\epsilon=0.04$, $\tau=2\cdot 10^{-3}$, $b=4\cdot 10^{-5}$.]{\label{subfig:parameter_3}{\input{vary_Re.tikz}}}
	%$\ $
%	  \subfloat[$h_{j}=2^{-j-6}$, $\varepsilon_{j}=9\, h_{j}/(2\sqrt{2}\cdot\textup{atanh}(0.95))$, $\tau_{1}=10^{-6}$, $\tau_{2}=10^{-7}$, $\tau_{3}=1.25\cdot 10^{-8}$ for $j=1,2,3$.]{\label{subfig:parameter_4}{\input{Pics_and_Data/Chapter_3/iter_smooth_semi-impl_all.tikz}}}
	\caption{Iteration numbers for the parameter study. The x-axis shows the time step and the y-axis the average number of FGMRES iterations per semismooth Newton step.}
	\label{fig:parameter_2}
\end{figure}

\begin{table}[!htbp]\footnotesize
  \centering
  \begin{tabular}{cc|rr|rrrrrrrrr}
      \toprule
       \multicolumn{2}{c|}{Simulation} & \multicolumn{2}{c|}{Newton} & \multicolumn{9}{c}{Parameters}\\\midrule
       Figure & Plot & Max & Avg & $N_{1}^{0}$ & $N_{2}^{0}$ & $Re$ & $\rho_{1}$ & $\hp$ & $\sigma$ & $\epsilon$ & $\tau$ & $b$\\\hline
      \midrule
      \ref{fig:parameter}\subref{subfig:parameter_1} & (\ref{fig:parameter:all:1}) & $6$ & $3$ & $6599$ & $26213$ & $35$ & $1000$ & $10^{4}$ & $15.60$ & $0.040$ & $2.000\cdot 10^{-3}$ & $4\cdot 10^{-5}$\\
        & (\ref{fig:parameter:all:2}) & $6$ & $2$ & $10399$ & $41413$ & $35$ & $1000$ & $10^{4}$ & $15.60$ & $0.020$ & $5.000\cdot 10^{-4}$ & $2\cdot 10^{-5}$\\
        & (\ref{fig:parameter:all:3}) & $6$ & $2$ & $17831$ & $71141$ & $35$ & $1000$ & $10^{4}$ & $15.60$ & $0.010$ & $1.250\cdot 10^{-4}$ & $1\cdot 10^{-5}$\\
        & (\ref{fig:parameter:all:4}) & $6$ & $2$ & $32527$ & $129925$ & $35$ & $1000$ & $10^{4}$ & $15.60$ & $0.005$ & $3.125\cdot 10^{-5}$ & $5\cdot 10^{-6}$\\\hline
      \ref{fig:parameter}\subref{subfig:parameter_2} & (\ref{fig:parameter:sigma:1}) & $8$ & $4$ & $6599$ & $26213$ & $35$ & $1000$ & $10^{6}$ & $0.02$ & $0.040$ & $2.000\cdot 10^{-3}$ & $4\cdot 10^{-5}$\\
        & (\ref{fig:parameter:sigma:2}) & $9$ & $5$ & $6599$ & $26213$ & $35$ & $1000$ & $10^{6}$ & $0.10$ & $0.040$ & $2.000\cdot 10^{-3}$ & $4\cdot 10^{-5}$\\
        & (\ref{fig:parameter:sigma:3}) & $9$ & $6$ & $6599$ & $26213$ & $35$ & $1000$ & $10^{6}$ & $1.00$ & $0.040$ & $2.000\cdot 10^{-3}$ & $4\cdot 10^{-5}$\\
        & (\ref{fig:parameter:sigma:4}) & $9$ & $6$ & $6599$ & $26213$ & $35$ & $1000$ & $10^{6}$ & $10.00$ & $0.040$ & $2.000\cdot 10^{-3}$ & $4\cdot 10^{-5}$\\
        & (\ref{fig:parameter:sigma:5}) & $10$ & $6$ & $6599$ & $26213$ & $35$ & $1000$ & $10^{6}$ & $90.00$ & $0.040$ & $2.000\cdot 10^{-3}$ & $4\cdot 10^{-5}$\\\hline
\ref{fig:parameter}\subref{subfig:parameter_3} & (\ref{fig:parameter:Re:1}) & $10$ & $6$ & $6599$ & $26213$ & $35$ & $1000$ & $10^{6}$ & $15.60$ & $0.040$ & $2.000\cdot 10^{-3}$ & $4\cdot 10^{-5}$\\
        & (\ref{fig:parameter:Re:2}) & $10$ & $7$ & $6599$ & $26213$ & $70$ & $2000$ & $10^{6}$ & $15.60$ & $0.040$ & $2.000\cdot 10^{-3}$ & $4\cdot 10^{-5}$\\
        & (\ref{fig:parameter:Re:3}) & $31$ & $7$ & $6599$ & $26213$ & $140$ & $4000$ & $10^{6}$ & $15.60$ & $0.040$ & $2.000\cdot 10^{-3}$ & $4\cdot 10^{-5}$\\
        & (\ref{fig:parameter:Re:4}) & $44$ & $8$ & $6599$ & $26213$ & $280$ & $8000$ & $10^{6}$ & $15.60$ & $0.040$ & $2.000\cdot 10^{-3}$ & $4\cdot 10^{-5}$\\
        & (\ref{fig:parameter:Re:5}) & $10$ & $7$ & $6599$ & $26213$ & $560$ & $16000$ & $10^{6}$ & $15.60$ & $0.040$ & $2.000\cdot 10^{-3}$ & $4\cdot 10^{-5}$\\\hline
 \ref{fig:parameter_2}\subref{subfig:parameter_4} & (\ref{fig:parameter:mob:1}) & $12$ & $7$ & $6599$ & $26213$ & $35$ & $1000$ & $10^{6}$ & $15.60$ & $0.040$ & $2.000\cdot 10^{-3}$ & $7\cdot 10^{-5}$\\
         & (\ref{fig:parameter:mob:2}) & $10$ & $6$ & $6599$ & $26213$ & $35$ & $1000$ & $10^{6}$ & $15.60$ & $0.040$ & $2.000\cdot 10^{-3}$ & $4\cdot 10^{-5}$\\
         & (\ref{fig:parameter:mob:3}) & $10$ & $7$ & $6599$ & $26213$ & $35$ & $1000$ & $10^{6}$ & $15.60$ & $0.040$ & $2.000\cdot 10^{-3}$ & $1\cdot 10^{-4}$\\
         & (\ref{fig:parameter:mob:4}) & $11$ & $7$ & $6599$ & $26213$ & $35$ & $1000$ & $10^{6}$ & $15.60$ & $0.040$ & $2.000\cdot 10^{-3}$ & $3\cdot 10^{-4}$\\\hline
 \ref{fig:parameter_2}\subref{subfig:parameter_5} & (\ref{fig:parameter:penalty:1}) & $6$ & $3$  & $6599$ & $26213$ & $35$ & $1000$ & $10^{4}$ & $15.60$ & $0.040$ & $2.000\cdot 10^{-3}$ & $4\cdot 10^{-5}$\\
         & (\ref{fig:parameter:penalty:2}) & $10$ & $6$  & $6599$ & $26213$ & $35$ & $1000$ & $10^{6}$ & $15.60$ & $0.040$ & $2.000\cdot 10^{-3}$ & $4\cdot 10^{-5}$\\
         & (\ref{fig:parameter:penalty:3}) & $18$ & $8$  & $6599$ & $26213$ & $35$ & $1000$ & $10^{8}$ & $15.60$ & $0.040$ & $2.000\cdot 10^{-3}$ & $4\cdot 10^{-5}$\\
         & (\ref{fig:parameter:penalty:4}) & $23$ & $8$  & $6599$ & $26213$ & $35$ & $1000$ & $10^{9}$ & $15.60$ & $0.040$ & $2.000\cdot 10^{-3}$ & $4\cdot 10^{-5}$\\ 
      \bottomrule
  \end{tabular}
  \caption{The maximum and average number of semismooth Newton iterations for each parameter test.}
  \label{table:parameter_Newton}
\end{table}

In summary, we observe a rather robust behavior of the iterative method with respect to 
changes of  the parameters of the system. Especially, the method is robust 
with respect to the parameter $\epsilon$, that heavily influences the number of degrees of freedom and therefore the size of the linear 
system, and the parameter $\hp$ for which we typically observe a severe increase of 
the condition number.

%%%%%%%%%%%%%%%%%%%%%%%%%%%%%%%%%%%%%%%%%%%%%%%%%%%

\subsection{A topology change}
In this section, we demonstrate the behavior of our preconditioner under topology changes. 
As test example, we use a quantitative benchmark for rising bubble dynamics; 
see the second benchmark in \cite{Hysing_Turek_quantitative_benchmark_computations_of_two_dimensional_bubble_dynamics}. 
It considers a bubble with a very low density compared to that of the surrounding fluid. 
The initial configuration is the same as in the previous benchmark example. The 
parameters are given as $\rho_{1}=1000$, $\rho_{2}=1$, $\eta_{1}=10$, $\eta_{2}=0.1$, $Re=35$, 
$N_{1}^{0}=4513$, $N_{2}^{0}=17929$, $\epsilon=0.04$, $\tau=0.002$, $b=4\cdot 10^{-5}$, $\sigma=1.24777$, $\hp=10^{6}$. 
As stated in \cite[p.~756]{AlaV12}, the decrease in surface tension causes the bubble to develop a more non-convex shape and thin filaments, 
which eventually break off. A simulation of this benchmark test example is
illustrated  in Figure~\ref{fig:Benchmark_2}.
\begin{figure}[!htbp]
  \centering
  \begin{tabular}{ccccccc}
    \includegraphics[trim=  7cm 4cm 17cm 5cm, clip,height=3cm]{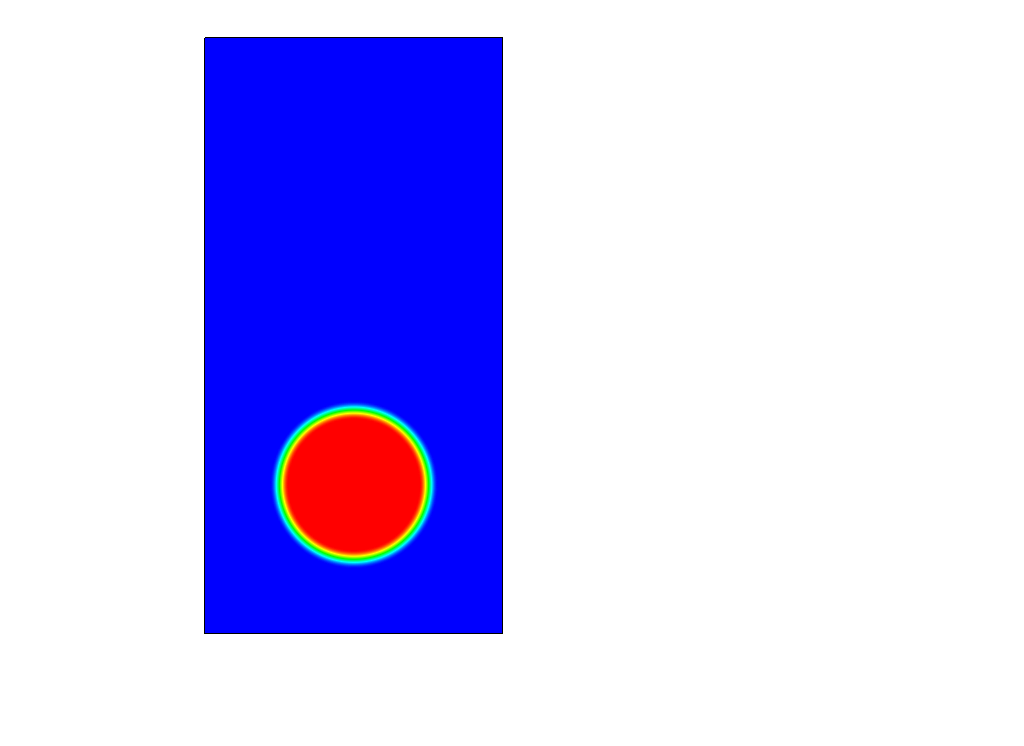} & &
    \includegraphics[trim=  7cm 4cm 17cm 5cm, clip,height=3cm]{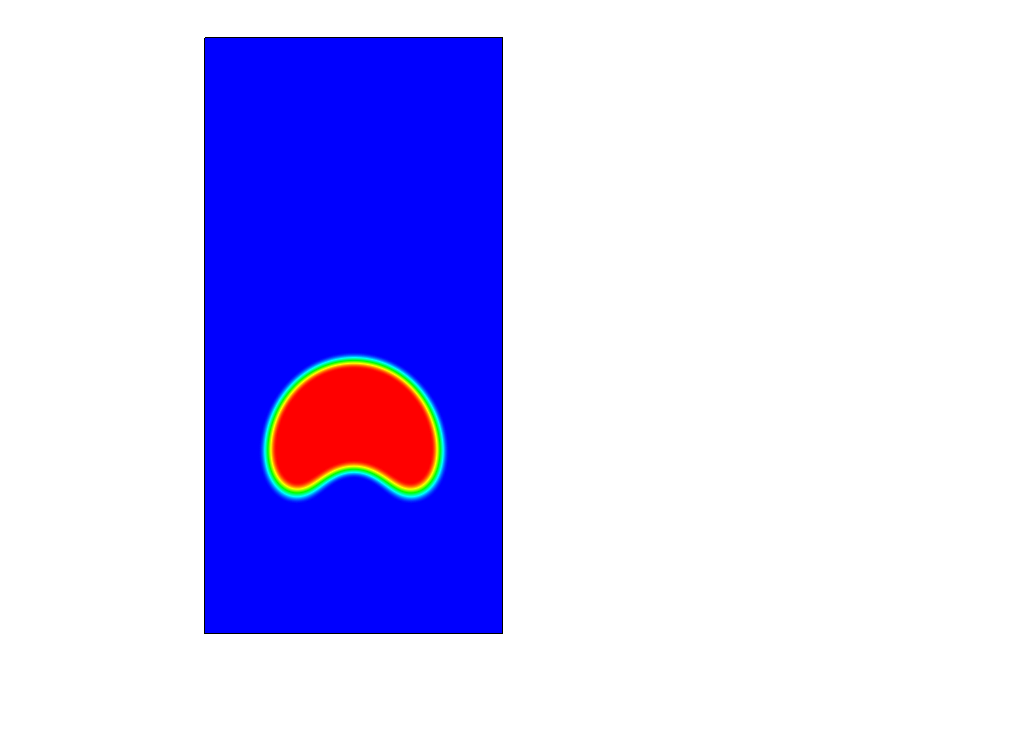} & &
    \includegraphics[trim=  7cm 4cm 17cm 5cm, clip,height=3cm]{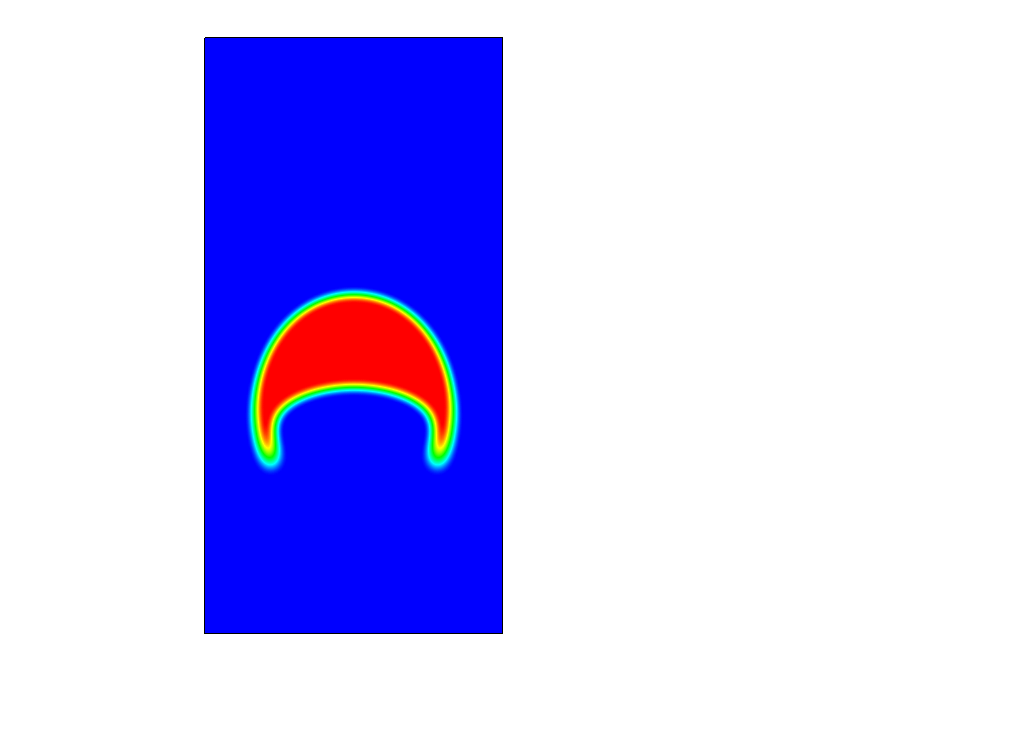} & &
    \includegraphics[trim=  7cm 4cm 17cm 5cm, clip,height=3cm]{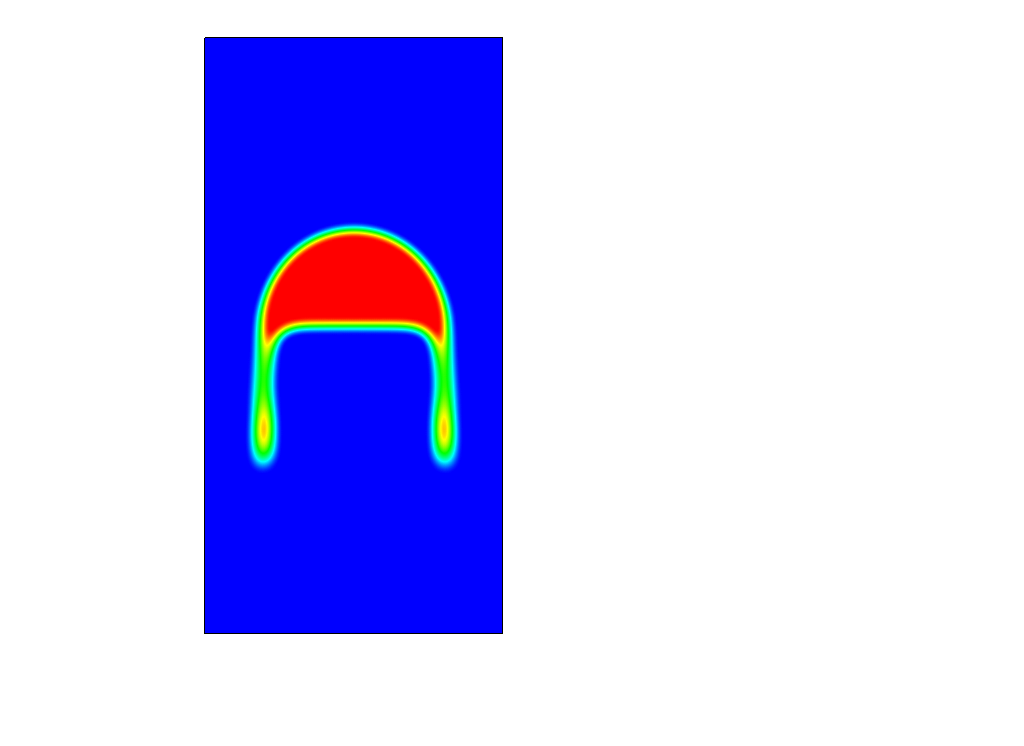}\\
    $t=0$ & & $t=1$ & & $t=2$ & & $t=3$
  \end{tabular}
  \caption{Simulation of a rising bubble under topology changes using a coupled Cahn--Hilliard Navier--Stokes model.}
  \label{fig:Benchmark_2}
\end{figure}

In this example, we set the tolerance on the relative residual norm for $S_{1}$ and $S_{2}$ to $10^{-6}$. 
Figure~\ref{fig:iterations_Bench2} demonstrates the average number of FGMRES iterations per semismooth Newton step during the 
time interval $[0,0.56]$, which consists of $280$ time steps. 
Since the crucial part is the period during the topology change, we show the
corresponding iteration numbers in Table~\ref{table:iterations_Bench2}. Note that for the results in Table \ref{table:iterations_Bench2}, 
we do not call \textsc{MATLAB}\textsuperscript{\textregistered} from C++.
Instead, we use Garcke et
al's~\cite{GarckeHinzeKahle_CHNS_AGG_linearStableTimeDisc} original C++
implementation, write the matrices and right-hand sides into text files, and
solve the problems in MATLAB using our developed preconditioner. There are two reasons for this procedure:
First, it takes a lot of time to get to the crucial instant of time using C++ with the MATLAB Engine API. 
Second, we observe abnormal terminations after long program runs. Hence, an important step for future research is an implementation of our 
preconditioner in C++ such that the MATLAB Engine API would be no longer required. 
Further, we  set the GMRES (inner solver) tolerance for the preconditioned relative residual to $10^{-2}$. 
Table~\ref{table:iterations_Bench2} shows that the FGMRES iteration numbers stay quite robust under topology changes.
As topology changes are only implicitly defined by a phase field approach, this is what we expected.
%Further, instead of using preconditioned GMRES as 
%inner iteration, we obtain here better results by simply replacing 
%$\boldsymbol{\hat{S}}$ in $\mathcal{P}_{\textup{out}}$ with $\mathcal{P}_{\textup{in}}$.
%Further, instead of using preconditioned GMRES as 
%inner iteration, we obtain here better results when using one step of a stationary iteration applied to $\boldsymbol{\hat{S}}$ 
%preconditioned with $\mathcal{P}_{\textup{in}}$. We get similar results by simply replacing 
%$\boldsymbol{\hat{S}}$ in $\mathcal{P}_{\textup{out}}$ with $\mathcal{P}_{\textup{in}}$.

\begin{figure}[!htbp]
	\centering
	\setlength\figureheight{4.7cm} 
	\setlength\figurewidth{7.5cm}
	\input{iter_Bench2.tikz}
\caption{Iteration numbers for the second benchmark example. 
The x-axis shows the time step and the y-axis the average number of FGMRES iterations per semismooth Newton step.}
	\label{fig:iterations_Bench2}
\end{figure}
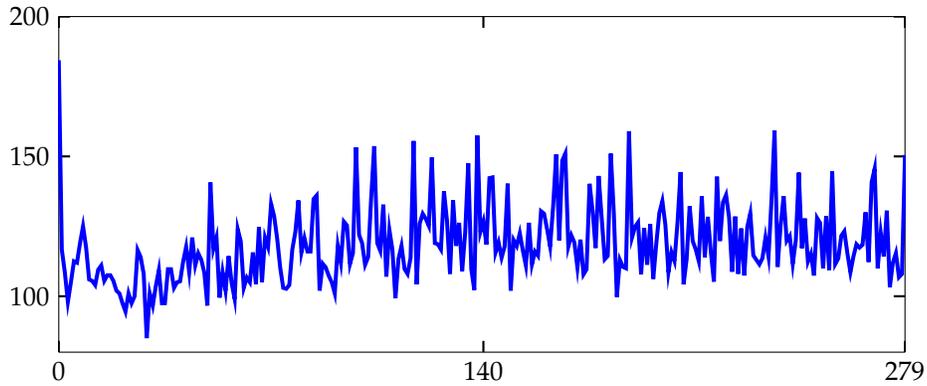

%\begin{table}[!htbp]\footnotesize
% \centering
% \begin{tabular}{r|r}
%     \toprule
%      Time & Iter\\\hline
%     \midrule
%     $1.9990$ & \\
%     $1.9995$ & \\
%     $2.0000$ & \\
%     $2.0005$ & \\
%     $2.0010$ & \\
%     $2.9985$ & \\
%     $2.9990$ & \\
%     $2.9995$ & \\
%     $3.0000$ & $219$\\
%     \bottomrule
% \end{tabular}
% \caption{Iteration numbers for the rising bubble under topology changes. The table shows the time with the average number of FGMRES iterations per semismooth Newton step.}
% \label{table:iterations_Bench2}
%\end{table}

\begin{table}[!htbp]\footnotesize
 \centering
 \begin{tabular}{c|rrrrrrrrr}
     \toprule
      $t$ & $1.9990$ & $1.9995$ & $2.0000$ & $2.0005$ & $2.0010$ & $2.9985$ & $2.9990$ & $2.9995$ & $3.0000$\\\midrule
      Iter & $179$ & $167$ & $127$ & $136$ & $144$ & $173$ & $202$ & $152$ & $151$\\
     \bottomrule
 \end{tabular}
 \caption{Iteration numbers for the rising bubble under topology chages. The table shows the time with the average number of FGMRES iterations per semismooth Newton step.}
 \label{table:iterations_Bench2}
\end{table}

\section{Conclusions}
\label{sec:concl}
%label identifier: *:concl:*
In this paper, we have investigated the efficient iterative solution of linear systems
that arise in the numerical simulation of a coupled 
Navier--Stokes Cahn--Hilliard system, like the model `H' from \cite{HohenbergHalperin1977} or its
generalization to fluids of different densities in \cite{AbelsGarckeGruen_CHNSmodell}.
As an example we have used the  energetically consistent discretization scheme
proposed in \cite{GarckeHinzeKahle_CHNS_AGG_linearStableTimeDisc}.

At the heart of this method lies the solution of large and sparse linear systems that arise in a 
semismooth Newton method. The systems to be solved are fully coupled. 
We have introduced and studied block-triangular preconditioners using efficient 
Schur complement approximations. 
For these approximations, we have used multilevel techniques, algebraic multigrid in our case. 
Extensive numerical experiments show a nearly parameter independent behavior of our developed preconditioners. 
Together with Garcke et 
al's \cite{GarckeHinzeKahle_CHNS_AGG_linearStableTimeDisc} adaptive spatial discretization scheme, this allows us
to perform three-dimensional experiments in an efficient way. This will be subject to future work.

%%%% Acknowledgments %%%%%%%%
\section*{Acknowledgments}
The authors would like to thank the anonymous referees for their helpful comments and 
suggestions.

%%%% Bibliography  %%%%%%%%%%
%\todo{change bibliographystyle to alpha}
\bibliographystyle{alpha}
\bibliography{bib_CHNS_Precond}

\end{document}

%% file: vary_all.tikz
% This file was created by matlab2tikz v0.4.7 running on MATLAB 9.0.
% Copyright (c) 2008--2014, Nico Schlömer <nico.schloemer@gmail.com>
% All rights reserved.
% Minimal pgfplots version: 1.3
% 
% The latest updates can be retrieved from
%   http://www.mathworks.com/matlabcentral/fileexchange/22022-matlab2tikz
% where you can also make suggestions and rate matlab2tikz.
%
\begin{tikzpicture}

\begin{axis}[%
width=1.5\figurewidth,
height=0.95\figureheight,
at={(0\figurewidth,0\figureheight)},
scale only axis,
separate axis lines,
every outer x axis line/.append style={black},
every x tick label/.append style={font=\color{black}},
xmin=0,
xmax=99,
major tick length=0.15cm,
minor tick length=0.075cm,
tick style={thick,color=black},
xtick={0,50,99},
every outer y axis line/.append style={black},
every y tick label/.append style={font=\color{black}},
ymin=40,
ymax=160,
ylabel style={align=center},
axis background/.style={fill=white},
legend style={legend cell align=left,align=left,draw=black},
legend pos=outer north east
]
\addplot [color=red,line width=1.5pt,dotted]
  table[row sep=crcr]{
0	72.4\\
1	63\\
2	59.75\\
3	63.3333\\
4	70.3333\\
5	79.3333\\
6	70.6667\\
7	68\\
8	68\\
9	66.5\\
10	58\\
11	59\\
12	67\\
13	67.5\\
14	59.5\\
15	62.5\\
16	58\\
17	76.3333\\
18	74.6667\\
19	79.6667\\
20	79\\
21	68.3333\\
22	65.6667\\
23	74.3333\\
24	70.6667\\
25	56\\
26	52.5\\
27	71.6667\\
28	72.3333\\
29	67\\
30	59.6667\\
31	66.4\\
32	63.3333\\
33	58.3333\\
34	65\\
35	66.3333\\
36	61.3333\\
37	57.3333\\
38	52.3333\\
39	56\\
40	63.3333\\
41	58\\
42	66\\
43	51.3333\\
44	63.6667\\
45	60.3333\\
46	56.3333\\
47	68.6667\\
48	63.6667\\
49	65\\
50	70\\
51	57.6667\\
52	68.8\\
53	50.6667\\
54	62\\
55	62.6667\\
56	64.3333\\
57	56.6667\\
58	72\\
59	67\\
60	68.3333\\
61	65.6\\
62	55\\
63	54\\
64	61\\
65	66.6667\\
66	68.6667\\
67	58\\
68	62.2\\
69	59\\
70	69.3333\\
71	64.75\\
72	56.6667\\
73	53\\
74	54.6667\\
75	53.3333\\
76	61.3333\\
77	63.3333\\
78	58.3333\\
79	68\\
80	61\\
81	73.4\\
82	76.3333\\
83	56.3333\\
84	58.4\\
85	70.75\\
86	64\\
87	63.75\\
88	69.5\\
89	63.5\\
90	59.75\\
91	64.6\\
92	70.75\\
93	73.5\\
94	64.2\\
95	68.3333\\
96	73\\
97	70\\
98	58.6667\\
99	69.75\\
};\label{fig:parameter:all:1}
\addlegendentry{set 1};

\addplot [color=blue,line width=1.5pt,dashdotted]
  table[row sep=crcr]{
0	113.8\\
1	101.5\\
2	113.75\\
3	91\\
4	76.5\\
5	70\\
6	68\\
7	61\\
8	76.5\\
9	56\\
10	61.5\\
11	62.5\\
12	59.5\\
13	56\\
14	58\\
15	56\\
16	57\\
17	52.5\\
18	51.5\\
19	52\\
20	52.5\\
21	52.5\\
22	54.5\\
23	56\\
24	56\\
25	54\\
26	55\\
27	55.5\\
28	56\\
29	56.5\\
30	49.5\\
31	48.5\\
32	49.5\\
33	46.5\\
34	46.5\\
35	45.5\\
36	45.5\\
37	45.5\\
38	51.5\\
39	46.5\\
40	47\\
41	48\\
42	50.5\\
43	46\\
44	50.5\\
45	49.5\\
46	49.5\\
47	47.5\\
48	48.5\\
49	55\\
50	53.5\\
51	54.5\\
52	53\\
53	53.5\\
54	54\\
55	53.5\\
56	53.5\\
57	54\\
58	54\\
59	55\\
60	54.5\\
61	56\\
62	51.5\\
63	55\\
64	55\\
65	54\\
66	55.5\\
67	56\\
68	53\\
69	53\\
70	53\\
71	53\\
72	54.5\\
73	54\\
74	54\\
75	53\\
76	53\\
77	56\\
78	54.5\\
79	53\\
80	52.5\\
81	60\\
82	67\\
83	68.5\\
84	50\\
85	50.5\\
86	48.5\\
87	47\\
88	47.5\\
89	44\\
90	55\\
91	50.5\\
92	54.5\\
93	54.5\\
94	51\\
95	53\\
96	52\\
97	53.5\\
98	67.5\\
99	53\\
};\label{fig:parameter:all:2}
\addlegendentry{set 2};

\addplot [color=green,line width=1.5pt,dashed]
  table[row sep=crcr]{0	122.8\\
1	105.833\\
2	99.25\\
3	89.3333\\
4	79\\
5	68\\
6	76.5\\
7	68\\
8	73.5\\
9	73.5\\
10	74.5\\
11	75\\
12	66\\
13	63\\
14	70\\
15	62.5\\
16	67.5\\
17	61.5\\
18	60\\
19	63.5\\
20	62\\
21	63\\
22	62\\
23	62\\
24	61\\
25	62\\
26	62\\
27	62\\
28	62\\
29	60\\
30	59\\
31	57.5\\
32	63\\
33	62\\
34	61\\
35	61.5\\
36	61.5\\
37	60.5\\
38	60.5\\
39	58\\
40	59\\
41	57.5\\
42	57.5\\
43	57.5\\
44	52\\
45	54\\
46	51.5\\
47	50.5\\
48	50\\
49	51\\
50	51\\
51	51\\
52	57.5\\
53	51\\
54	53.5\\
55	52.5\\
56	51\\
57	51.5\\
58	56\\
59	55\\
60	52\\
61	57.5\\
62	56.5\\
63	57.5\\
64	56.5\\
65	61.5\\
66	58\\
67	61\\
68	57\\
69	56.5\\
70	56.5\\
71	56\\
72	55.5\\
73	55\\
74	56\\
75	55\\
76	58.5\\
77	55\\
78	57.5\\
79	55\\
80	57\\
81	58.5\\
82	52\\
83	54.5\\
84	58\\
85	53\\
86	52\\
87	54.5\\
88	57\\
89	52\\
90	52\\
91	56\\
92	58.5\\
93	51.5\\
94	52.5\\
95	58.5\\
96	61.5\\
97	56\\
98	58.5\\
99	56\\
};\label{fig:parameter:all:3}
\addlegendentry{set 3};

\addplot [color=cyan,line width=1.5pt,solid]
  table[row sep=crcr]{0	150.833\\
1	110.4\\
2	97.6667\\
3	85\\
4	100\\
5	81\\
6	76\\
7	76\\
8	103\\
9	128.5\\
10	86\\
11	77\\
12	74\\
13	74\\
14	70\\
15	70.5\\
16	67\\
17	70.5\\
18	63\\
19	66.5\\
20	66\\
21	67.5\\
22	65.5\\
23	65.5\\
24	66.5\\
25	65.5\\
26	65.5\\
27	65\\
28	65.5\\
29	65.5\\
30	64\\
31	63\\
32	63\\
33	61.5\\
34	61.5\\
35	60.5\\
36	60\\
37	60\\
38	59\\
39	58.5\\
40	57\\
41	56\\
42	57.5\\
43	57\\
44	58.5\\
45	57\\
46	58\\
47	58.5\\
48	58.5\\
49	58.5\\
50	59.5\\
51	57\\
52	59.5\\
53	59.5\\
54	58\\
55	55.5\\
56	57\\
57	59.5\\
58	55.5\\
59	60.5\\
60	60\\
61	60.5\\
62	61\\
63	55.5\\
64	60\\
65	56.5\\
66	56.5\\
67	56.5\\
68	56.5\\
69	60\\
70	56.5\\
71	56.5\\
72	60.5\\
73	56.5\\
74	56.5\\
75	56.5\\
76	56.5\\
77	56.5\\
78	56.5\\
79	61\\
80	61.5\\
81	56.5\\
82	56.5\\
83	56.5\\
84	59.5\\
85	56.5\\
86	56\\
87	56\\
88	56\\
89	56\\
90	55.5\\
91	56\\
92	55.5\\
93	55\\
94	55.5\\
95	55\\
96	56.5\\
97	55\\
98	55\\
99	63\\
};\label{fig:parameter:all:4}
\addlegendentry{set 4};

\end{axis}
\end{tikzpicture}%

%% file: vary_sigma.tikz
% This file was created by matlab2tikz v0.4.7 running on MATLAB 9.0.
% Copyright (c) 2008--2014, Nico Schlömer <nico.schloemer@gmail.com>
% All rights reserved.
% Minimal pgfplots version: 1.3
% 
% The latest updates can be retrieved from
%   http://www.mathworks.com/matlabcentral/fileexchange/22022-matlab2tikz
% where you can also make suggestions and rate matlab2tikz.
% 
\begin{tikzpicture}

\begin{axis}[%
width=1.5\figurewidth,
height=0.95\figureheight,
at={(0\figurewidth,0\figureheight)},
scale only axis,
separate axis lines,
every outer x axis line/.append style={black},
every x tick label/.append style={font=\color{black}},
xmin=0,
xmax=99,
major tick length=0.15cm,
minor tick length=0.075cm,
tick style={thick,color=black},
xtick={0,50,99},
every outer y axis line/.append style={black},
every y tick label/.append style={font=\color{black}},
ymin=40,
ymax=80,
ylabel style={align=center},
axis background/.style={fill=white},
legend style={legend cell align=left,align=left,draw=black},
legend pos=outer north east
]
\addplot [color=red,line width=1.5pt,dotted]
  table[row sep=crcr]{
0	53.5\\
1	61.1667\\
2	65.25\\
3	54\\
4	53\\
5	54\\
6	60\\
7	62.25\\
8	62.75\\
9	67.2\\
10	54\\
11	54\\
12	53.5\\
13	54\\
14	54\\
15	54\\
16	54\\
17	55\\
18	55\\
19	54\\
20	54\\
21	54.5\\
22	53\\
23	57\\
24	53\\
25	54\\
26	54\\
27	51\\
28	64.2\\
29	49\\
30	46\\
31	63\\
32	62\\
33	49.5\\
34	52\\
35	50\\
36	62.25\\
37	48\\
38	63.4\\
39	46.5\\
40	45.5\\
41	57\\
42	48.75\\
43	43.5\\
44	55\\
45	50.5\\
46	52.8\\
47	46.2\\
48	62\\
49	55.3333\\
50	64.2\\
51	64\\
52	54.2\\
53	55.6\\
54	56.3333\\
55	56.4\\
56	63.75\\
57	58.5\\
58	57\\
59	51\\
60	51\\
61	53.2857\\
62	63.75\\
63	57.2\\
64	50.5\\
65	50\\
66	59.6\\
67	55.4\\
68	50.4\\
69	59\\
70	57.6667\\
71	56.75\\
72	59.3333\\
73	57\\
74	54\\
75	56\\
76	52.5\\
77	58.2\\
78	55.3333\\
79	59.4\\
80	63.6\\
81	59.8\\
82	59\\
83	48\\
84	57.4\\
85	45\\
86	65\\
87	65.4\\
88	59.2\\
89	53.8\\
90	64.2\\
91	53.5\\
92	57.75\\
93	60.625\\
94	59\\
95	58.8333\\
96	67.4\\
97	53\\
98	59.8\\
99	72.25\\
};\label{fig:parameter:sigma:1}
\addlegendentry{$\sigma=0.02$};

\addplot [color=blue,line width=1.5pt,dashdotted]
  table[row sep=crcr]{
0	52\\
1	57.2857\\
2	60.2\\
3	49.5\\
4	52.5\\
5	53\\
6	53\\
7	53\\
8	53\\
9	53\\
10	53\\
11	52\\
12	53\\
13	54\\
14	57.1429\\
15	55\\
16	51.5\\
17	52.5\\
18	49.5\\
19	62.6667\\
20	49.1429\\
21	60\\
22	52.5\\
23	51.5\\
24	52\\
25	63\\
26	57.8333\\
27	61.6667\\
28	54.8571\\
29	57.5\\
30	55.1667\\
31	67.1667\\
32	59.3333\\
33	45.5\\
34	45.5\\
35	55.4286\\
36	61\\
37	50.3333\\
38	60.8333\\
39	49.5\\
40	58.4\\
41	47.3333\\
42	57\\
43	52.5\\
44	49.6667\\
45	60\\
46	51\\
47	67.5\\
48	61.5\\
49	62.6667\\
50	59.6667\\
51	63\\
52	59\\
53	55.5\\
54	59.25\\
55	57.6667\\
56	53\\
57	57.4444\\
58	57\\
59	62.8\\
60	66\\
61	60.75\\
62	61\\
63	66.6667\\
64	53.5714\\
65	59.2857\\
66	62.2857\\
67	55\\
68	62\\
69	63\\
70	55.6667\\
71	56.6667\\
72	53.6667\\
73	55\\
74	56.8571\\
75	64.2\\
76	65\\
77	58.8333\\
78	59.75\\
79	57.1667\\
80	62.3333\\
81	53.1667\\
82	62\\
83	62.8\\
84	60.3333\\
85	61.3333\\
86	60.4286\\
87	54.5714\\
88	57.75\\
89	64.5714\\
90	57.8571\\
91	59.5714\\
92	60.4\\
93	63.3333\\
94	55\\
95	61.25\\
96	57.4286\\
97	54.75\\
98	54\\
99	57.7143\\
};\label{fig:parameter:sigma:2}
\addlegendentry{$\sigma=0.1$};

\addplot [color=green,line width=1.5pt,dashed]
  table[row sep=crcr]{0	78\\
1	55.6667\\
2	71.3333\\
3	59.3333\\
4	52.6667\\
5	60\\
6	60.25\\
7	54.7778\\
8	51.8571\\
9	53.6667\\
10	70\\
11	52\\
12	51.5\\
13	58.6667\\
14	54\\
15	60.6667\\
16	57\\
17	61.5\\
18	66.25\\
19	58.3333\\
20	54\\
21	60.3333\\
22	47.5\\
23	56.1429\\
24	61.6\\
25	64.6667\\
26	62.2857\\
27	61.4\\
28	62.6667\\
29	66.6667\\
30	54.3333\\
31	58\\
32	62.8\\
33	61.25\\
34	61.125\\
35	56.4\\
36	53\\
37	54\\
38	60.1667\\
39	54.6\\
40	62.4286\\
41	60.5\\
42	66.6667\\
43	54.6667\\
44	61.8889\\
45	61\\
46	51.75\\
47	54.3333\\
48	58.5\\
49	57.2222\\
50	65\\
51	59\\
52	59.5\\
53	57.75\\
54	60.1429\\
55	54.4\\
56	55.3333\\
57	61.75\\
58	58.5\\
59	75.75\\
60	56.125\\
61	53.6\\
62	60\\
63	54\\
64	61.6667\\
65	54.375\\
66	58.3333\\
67	54.8889\\
68	55.8571\\
69	62.8\\
70	59.2857\\
71	60.2222\\
72	57.25\\
73	61.3333\\
74	60.3333\\
75	56.8571\\
76	50.75\\
77	65.125\\
78	63.5\\
79	60.875\\
80	58.5\\
81	61.25\\
82	67.8571\\
83	58.4444\\
84	56.25\\
85	55.625\\
86	59.75\\
87	62\\
88	59\\
89	59.5\\
90	59.6\\
91	66.5\\
92	58.125\\
93	55.5\\
94	71.75\\
95	60.6\\
96	57.7778\\
97	56.6667\\
98	57.8889\\
99	54.1429\\
};\label{fig:parameter:sigma:3}
\addlegendentry{$\sigma=1$};

\addplot [color=cyan,line width=1.5pt,solid]
  table[row sep=crcr]{0	60.4444\\
1	59.2\\
2	59.1\\
3	65.25\\
4	58.6667\\
5	59.6667\\
6	63.3333\\
7	62.75\\
8	56.25\\
9	55.5\\
10	64.75\\
11	58.3333\\
12	62\\
13	59.5\\
14	61.3333\\
15	53.25\\
16	63.25\\
17	57.5\\
18	57.5\\
19	57.6667\\
20	57.5\\
21	56\\
22	66.5\\
23	59.5\\
24	54.6\\
25	61.8\\
26	56\\
27	61.8889\\
28	57\\
29	58.5\\
30	57.3333\\
31	68\\
32	56\\
33	64\\
34	52.4\\
35	56.8889\\
36	51.875\\
37	54.875\\
38	59.6\\
39	69\\
40	57.5714\\
41	50.5556\\
42	54.2857\\
43	58.8571\\
44	58.7143\\
45	51.5556\\
46	53\\
47	59.8\\
48	61.6\\
49	56\\
50	61\\
51	49.3333\\
52	60.1111\\
53	61.3333\\
54	71.2857\\
55	60.1111\\
56	56\\
57	64.25\\
58	70.5\\
59	59.1\\
60	63.1667\\
61	55.375\\
62	54.125\\
63	55.7143\\
64	59.375\\
65	59.8\\
66	63.25\\
67	56.6667\\
68	60.5\\
69	56.75\\
70	55.4\\
71	54\\
72	62.1667\\
73	54.5556\\
74	59.5\\
75	58.4\\
76	58.8\\
77	60.875\\
78	56.3333\\
79	51.7778\\
80	55.1429\\
81	53.2\\
82	56\\
83	60\\
84	58\\
85	57.8333\\
86	53\\
87	55.5\\
88	56.8\\
89	52.4444\\
90	60.3333\\
91	61\\
92	53.7778\\
93	57.7778\\
94	58.25\\
95	54.3333\\
96	56.1667\\
97	57.5\\
98	57.5\\
99	68\\
};\label{fig:parameter:sigma:4}
\addlegendentry{$\sigma=10$};

\addplot [color=yellow,line width=1.5pt,dotted]
  table[row sep=crcr]{0	64.5455\\
1	57.4286\\
2	65.8333\\
3	66\\
4	72\\
5	55.4\\
6	55.4\\
7	58.4\\
8	53.25\\
9	55\\
10	55\\
11	57.6667\\
12	46\\
13	52.3333\\
14	47.6667\\
15	44.3333\\
16	49.8\\
17	56.5714\\
18	49.6\\
19	55.75\\
20	61\\
21	56.2\\
22	50.8\\
23	54.5\\
24	53\\
25	59\\
26	51.1111\\
27	52.8889\\
28	53.125\\
29	64.5\\
30	65.75\\
31	59.75\\
32	57.25\\
33	50.2\\
34	54.3333\\
35	51.3333\\
36	52.5\\
37	51.5\\
38	51.875\\
39	51.4\\
40	58.75\\
41	51.25\\
42	50.75\\
43	49.6667\\
44	61.1667\\
45	54.8\\
46	54.3333\\
47	52.6\\
48	55.8571\\
49	51.6\\
50	54.4444\\
51	50.8889\\
52	49.7143\\
53	50\\
54	58.5\\
55	48.8333\\
56	53.1\\
57	50.625\\
58	55.1111\\
59	58.7\\
60	62.625\\
61	53\\
62	53\\
63	52.8\\
64	55.1111\\
65	56.7\\
66	58.6\\
67	58.5\\
68	53.6\\
69	58.3333\\
70	51.8571\\
71	75.3846\\
72	55.5556\\
73	52.8889\\
74	55.2222\\
75	52.6667\\
76	55.8571\\
77	51.8\\
78	52.5556\\
79	57.9\\
80	52.875\\
81	57\\
82	51\\
83	52.7778\\
84	52\\
85	51.7\\
86	54.4444\\
87	53.8333\\
88	53.8\\
89	56\\
90	53.375\\
91	52.375\\
92	57.25\\
93	52\\
94	51\\
95	52.4\\
96	53.6\\
97	57.5\\
98	61.1\\
99	51.6667\\
};\label{fig:parameter:sigma:5}
\addlegendentry{$\sigma=90$};

\end{axis}
\end{tikzpicture}%

%% file: vary_Re.tikz
% This file was created by matlab2tikz v0.4.7 running on MATLAB 9.0.
% Copyright (c) 2008--2014, Nico Schlömer <nico.schloemer@gmail.com>
% All rights reserved.
% Minimal pgfplots version: 1.3
% 
% The latest updates can be retrieved from
%   http://www.mathworks.com/matlabcentral/fileexchange/22022-matlab2tikz
% where you can also make suggestions and rate matlab2tikz.
% 
\begin{tikzpicture}

\begin{axis}[%
width=1.5\figurewidth,
height=0.95\figureheight,
at={(0\figurewidth,0\figureheight)},
scale only axis,
separate axis lines,
every outer x axis line/.append style={black},
every x tick label/.append style={font=\color{black}},
xmin=0,
xmax=99,
major tick length=0.15cm,
minor tick length=0.075cm,
tick style={thick,color=black},
xtick={0,50,99},
every outer y axis line/.append style={black},
every y tick label/.append style={font=\color{black}},
ymin=40,
ymax=120,
ylabel style={align=center},
axis background/.style={fill=white},
legend style={legend cell align=left,align=left,draw=black},
legend pos=outer north east
]
\addplot [color=red,line width=1.5pt,dotted]
  table[row sep=crcr]{
0	64.25\\
1	58.5556\\
2	58.5714\\
3	60\\
4	57.6667\\
5	59.5\\
6	62.4\\
7	53.5\\
8	52.25\\
9	57.6667\\
10	59\\
11	61\\
12	72.3333\\
13	61.75\\
14	65.8\\
15	51\\
16	56.3333\\
17	58.75\\
18	57.25\\
19	57.25\\
20	52\\
21	48.5\\
22	52.4\\
23	54.25\\
24	50.4\\
25	58\\
26	58.6667\\
27	57\\
28	51.5\\
29	58.5\\
30	55.4286\\
31	53.5\\
32	50.5\\
33	55.75\\
34	58.8\\
35	56.5\\
36	51.2222\\
37	59.5556\\
38	57.25\\
39	64.75\\
40	46.2222\\
41	50.875\\
42	53\\
43	60.5\\
44	57.6\\
45	59.3333\\
46	60.4286\\
47	61.625\\
48	56.5\\
49	56.6\\
50	53.6\\
51	51.4444\\
52	54.5714\\
53	51.8333\\
54	59.3333\\
55	62.4444\\
56	61\\
57	52\\
58	50\\
59	50.4\\
60	63.75\\
61	60.8\\
62	60.8333\\
63	56.5\\
64	55\\
65	55.7143\\
66	55.6\\
67	54.3333\\
68	55.1667\\
69	48.5\\
70	47.4\\
71	50.2\\
72	51.625\\
73	53.8333\\
74	52.4444\\
75	59.3\\
76	54.6\\
77	49.7778\\
78	56.8571\\
79	54.875\\
80	55.7778\\
81	64.1429\\
82	67.2\\
83	57.375\\
84	55.4\\
85	54.5556\\
86	64.875\\
87	52.6\\
88	60.625\\
89	60.125\\
90	60.1\\
91	54\\
92	60.75\\
93	60\\
94	60.7778\\
95	59.4286\\
96	55.8\\
97	55.5556\\
98	53.1429\\
99	52\\
};\label{fig:parameter:Re:1}
\addlegendentry{$Re=35$};

\addplot [color=blue,line width=1.5pt,dashdotted]
  table[row sep=crcr]{
0	70.875\\
1	63.8889\\
2	61.4286\\
3	56.6667\\
4	59.3333\\
5	61.5\\
6	63.6\\
7	59\\
8	57.75\\
9	56\\
10	56.3333\\
11	64.6667\\
12	53.6667\\
13	57.5\\
14	54\\
15	65.75\\
16	56.8333\\
17	60.5\\
18	66.2\\
19	61.5\\
20	51.75\\
21	51.6\\
22	51\\
23	55.6\\
24	63.75\\
25	57\\
26	56\\
27	60.7778\\
28	63.875\\
29	65.2\\
30	55.25\\
31	51.5\\
32	55.8\\
33	64.25\\
34	60\\
35	58.75\\
36	50.8571\\
37	65.25\\
38	53.1111\\
39	56.1429\\
40	60.4\\
41	61.75\\
42	59.8889\\
43	59.1111\\
44	59.1429\\
45	64\\
46	62.75\\
47	56.375\\
48	55.4444\\
49	60.5\\
50	61.3333\\
51	62.8889\\
52	71.4286\\
53	63.6\\
54	54.8333\\
55	49\\
56	59.4\\
57	63.875\\
58	63.2\\
59	62.125\\
60	64.25\\
61	61.8333\\
62	57\\
63	60\\
64	58.7143\\
65	62\\
66	54.4\\
67	59.1111\\
68	58.7143\\
69	65.6\\
70	64.8889\\
71	52.125\\
72	59.8889\\
73	58.875\\
74	62.1111\\
75	54.2222\\
76	62.2222\\
77	65.7143\\
78	62.5\\
79	60.3333\\
80	61.25\\
81	55.7\\
82	59\\
83	57.1111\\
84	62.5\\
85	60.9\\
86	55\\
87	63\\
88	53.1667\\
89	62.4444\\
90	54.625\\
91	55.3333\\
92	61.8889\\
93	57.8333\\
94	63.25\\
95	61.8\\
96	65.625\\
97	67.5\\
98	57.7143\\
99	60.7778\\
};\label{fig:parameter:Re:2}
\addlegendentry{$Re=70$};

\addplot [color=green,line width=1.5pt,dashed]
  table[row sep=crcr]{0	74.125\\
1	73.1111\\
2	70.7143\\
3	81\\
4	66\\
5	69.5\\
6	64\\
7	66\\
8	64.25\\
9	59\\
10	59.3333\\
11	63.6667\\
12	56.3333\\
13	62\\
14	57.25\\
15	62.4\\
16	65.375\\
17	72\\
18	70.8\\
19	57.5\\
20	55\\
21	53.75\\
22	59.2\\
23	60\\
24	60.3333\\
25	61.6667\\
26	65.5556\\
27	64\\
28	71.75\\
29	71\\
30	70.8\\
31	66.4\\
32	75\\
33	67.125\\
34	56.6667\\
35	56.8571\\
36	69.6\\
37	58.9\\
38	65.1429\\
39	68.6\\
40	67.8889\\
41	59.1111\\
42	65.875\\
43	70.4\\
44	70\\
45	75.3333\\
46	73.3333\\
47	58.7778\\
48	83.7667\\
49	69.75\\
50	62.625\\
51	66.8\\
52	60.6\\
53	67.8\\
54	77.4444\\
55	68.4286\\
56	71.2\\
57	68.4444\\
58	78.1667\\
59	69.75\\
60	61.9\\
61	59.3\\
62	70\\
63	74.8\\
64	68.4\\
65	67\\
66	63\\
67	65.6667\\
68	76\\
69	63.8889\\
70	59.7\\
71	68.7778\\
72	66.6667\\
73	59.7778\\
74	65.8571\\
75	62.1667\\
76	64.6667\\
77	61.8889\\
78	72\\
79	64\\
80	73.1111\\
81	62.25\\
82	78.375\\
83	58.8889\\
84	63.25\\
85	73.6667\\
86	84.5484\\
87	70\\
88	68.125\\
89	59.625\\
90	69.5556\\
91	74.4444\\
92	75.5\\
93	81.2\\
94	74.8571\\
95	79.5\\
96	59.7\\
97	61.1111\\
98	61.6667\\
99	69.75\\
};\label{fig:parameter:Re:3}
\addlegendentry{$Re=140$};

\addplot [color=cyan,line width=1.5pt,solid]
  table[row sep=crcr]{0	82.625\\
1	71.1111\\
2	67.4286\\
3	67\\
4	69.3333\\
5	63.5\\
6	64.6\\
7	66.75\\
8	75.25\\
9	67.6667\\
10	87\\
11	82.3333\\
12	77.3333\\
13	75.75\\
14	65.5\\
15	70.75\\
16	67.5\\
17	70.6\\
18	95.44\\
19	67.5\\
20	62.2\\
21	66.25\\
22	67\\
23	65.75\\
24	76\\
25	74.2222\\
26	74.875\\
27	72.1429\\
28	73.8\\
29	82.069\\
30	76.2\\
31	119.477\\
32	81.1111\\
33	72\\
34	65.375\\
35	72.8\\
36	73.5\\
37	66.2222\\
38	77.1667\\
39	72.5556\\
40	83.5714\\
41	82.1111\\
42	75.4\\
43	83.5\\
44	80\\
45	71.75\\
46	58.6667\\
47	79.8889\\
48	80.25\\
49	70.7778\\
50	74.8333\\
51	76.5\\
52	68.8333\\
53	80.3333\\
54	68.875\\
55	76.7778\\
56	80.625\\
57	75.4\\
58	83.25\\
59	65.5\\
60	68.1\\
61	74.2857\\
62	75.3333\\
63	72.6\\
64	75.1111\\
65	77.4\\
66	82.25\\
67	67.7778\\
68	62.6667\\
69	75.4\\
70	80.625\\
71	66.8\\
72	67.625\\
73	76.4286\\
74	74\\
75	78.625\\
76	65.25\\
77	76.3333\\
78	73.7\\
79	60.625\\
80	78.125\\
81	67.7\\
82	75.125\\
83	67.875\\
84	73.5\\
85	80.5\\
86	81.2222\\
87	70.5556\\
88	72.2857\\
89	75.1111\\
90	68.625\\
91	84.6667\\
92	74.4444\\
93	84.1429\\
94	66.1111\\
95	62.6667\\
96	62.2222\\
97	74.5\\
98	83.8889\\
99	75.5556\\
};\label{fig:parameter:Re:4}
\addlegendentry{$Re=280$};

\addplot [color=yellow,line width=1.5pt,dotted]
  table[row sep=crcr]{0	98.875\\
1	77.5556\\
2	88.1429\\
3	73\\
4	75.3333\\
5	71.25\\
6	67.6\\
7	83.25\\
8	91.25\\
9	80\\
10	95\\
11	94.3333\\
12	89\\
13	89.75\\
14	81.4\\
15	82.5\\
16	87\\
17	78.2\\
18	81\\
19	89.5\\
20	70.4\\
21	71.75\\
22	79.5\\
23	87.75\\
24	82\\
25	86.7778\\
26	74\\
27	66.4286\\
28	75.6\\
29	81\\
30	88.8\\
31	87.8\\
32	86.625\\
33	75\\
34	73.625\\
35	83.3333\\
36	87.2\\
37	66.625\\
38	76.8\\
39	82.25\\
40	80\\
41	75\\
42	77.8\\
43	85\\
44	99.6667\\
45	71.5556\\
46	78.4444\\
47	84.1111\\
48	79.2857\\
49	73.5556\\
50	83.2\\
51	78.4\\
52	76.7\\
53	76.8889\\
54	78.1429\\
55	88\\
56	87\\
57	93.4\\
58	63.1111\\
59	67.2\\
60	66.5556\\
61	80.8571\\
62	84.2\\
63	86.6\\
64	75.375\\
65	94.75\\
66	77.2\\
67	70.6667\\
68	74.9\\
69	84.125\\
70	85\\
71	78.2222\\
72	78.4286\\
73	81.25\\
74	88.875\\
75	72.6667\\
76	87.2857\\
77	76.9\\
78	61.8889\\
79	75.625\\
80	80.125\\
81	76.1111\\
82	80\\
83	77.875\\
84	85\\
85	80.4444\\
86	83.625\\
87	73.5714\\
88	81.1111\\
89	74.5556\\
90	78.1111\\
91	82.5\\
92	89.6667\\
93	70.5556\\
94	69\\
95	61.75\\
96	76.8\\
97	84\\
98	86.1111\\
99	74.8889\\
};\label{fig:parameter:Re:5}
\addlegendentry{$Re=560$};

\end{axis}
\end{tikzpicture}%

%% file: vary_mobility.tikz
% This file was created by matlab2tikz v0.4.7 running on MATLAB 9.0.
% Copyright (c) 2008--2014, Nico Schlömer <nico.schloemer@gmail.com>
% All rights reserved.
% Minimal pgfplots version: 1.3
% 
% The latest updates can be retrieved from
%   http://www.mathworks.com/matlabcentral/fileexchange/22022-matlab2tikz
% where you can also make suggestions and rate matlab2tikz.
%
\begin{tikzpicture}

\begin{axis}[%
width=1.5\figurewidth,
height=0.95\figureheight,
at={(0\figurewidth,0\figureheight)},
scale only axis,
separate axis lines,
every outer x axis line/.append style={black},
every x tick label/.append style={font=\color{black}},
xmin=0,
xmax=99,
major tick length=0.15cm,
minor tick length=0.075cm,
tick style={thick,color=black},
xtick={0,50,99},
every outer y axis line/.append style={black},
every y tick label/.append style={font=\color{black}},
ymin=35,
ymax=75,
ylabel style={align=center},
axis background/.style={fill=white},
legend style={legend cell align=left,align=left,draw=black},
legend pos=outer north east
]
\addplot [color=red,line width=1.5pt,dotted]
  table[row sep=crcr]{0	62.7143\\
1	50.4286\\
2	56\\
3	51.6\\
4	50.8\\
5	48\\
6	45.6667\\
7	57.6667\\
8	54.3333\\
9	47.3333\\
10	52.25\\
11	47.25\\
12	47.25\\
13	51\\
14	54.6667\\
15	55.3333\\
16	49.6667\\
17	52.25\\
18	50.6\\
19	48.6\\
20	44\\
21	46.6667\\
22	52.75\\
23	47.5\\
24	50.3333\\
25	48.4286\\
26	48.375\\
27	51.4286\\
28	47.5\\
29	48.5\\
30	46.5\\
31	54.9\\
32	49.4286\\
33	49.6667\\
34	50.875\\
35	51\\
36	46.875\\
37	50.5\\
38	49.6\\
39	48.1\\
40	48.4286\\
41	54\\
42	47.8\\
43	50.2\\
44	47\\
45	50\\
46	56.75\\
47	54.7778\\
48	51.1429\\
49	47.9167\\
50	52.2222\\
51	52.4286\\
52	48.6667\\
53	48.5714\\
54	48\\
55	50.7\\
56	51.875\\
57	54.5\\
58	57\\
59	49\\
60	47.25\\
61	46\\
62	52.3333\\
63	54\\
64	49.9\\
65	51.6\\
66	49.8889\\
67	51.3333\\
68	52\\
69	44.8\\
70	49.5\\
71	48.7143\\
72	51.3333\\
73	51.5\\
74	51.5\\
75	52.625\\
76	57.375\\
77	48.4444\\
78	52\\
79	49.2857\\
80	49.8\\
81	52.25\\
82	51.3333\\
83	50.3333\\
84	52.1111\\
85	53.3333\\
86	50.2222\\
87	53\\
88	53.75\\
89	52.1\\
90	54.125\\
91	51.75\\
92	54.3333\\
93	56.5714\\
94	53.8889\\
95	60.7143\\
96	51.375\\
97	53.375\\
98	49.2\\
99	49.1111\\
};\label{fig:parameter:mob:1}
\addlegendentry{$b=7\cdot 10^{-5}$};

\addplot [color=blue,line width=1.5pt,dashdotted]
  table[row sep=crcr]{0	64.25\\
1	58.5556\\
2	58.5714\\
3	60\\
4	57.6667\\
5	59.5\\
6	62.4\\
7	53.5\\
8	52.25\\
9	57.6667\\
10	59\\
11	61\\
12	72.3333\\
13	61.75\\
14	65.8\\
15	51\\
16	56.3333\\
17	58.75\\
18	57.25\\
19	57.25\\
20	52\\
21	48.5\\
22	52.4\\
23	54.25\\
24	50.4\\
25	58\\
26	58.6667\\
27	57\\
28	51.5\\
29	58.5\\
30	55.4286\\
31	53.5\\
32	50.5\\
33	55.75\\
34	58.8\\
35	56.5\\
36	51.2222\\
37	59.5556\\
38	57.25\\
39	64.75\\
40	46.2222\\
41	50.875\\
42	53\\
43	60.5\\
44	57.6\\
45	59.3333\\
46	60.4286\\
47	61.625\\
48	56.5\\
49	56.6\\
50	53.6\\
51	51.4444\\
52	54.5714\\
53	51.8333\\
54	59.3333\\
55	62.4444\\
56	61\\
57	52\\
58	50\\
59	50.4\\
60	63.75\\
61	60.8\\
62	60.8333\\
63	56.5\\
64	55\\
65	55.7143\\
66	55.6\\
67	54.3333\\
68	55.1667\\
69	48.5\\
70	47.4\\
71	50.2\\
72	51.625\\
73	53.8333\\
74	52.4444\\
75	59.3\\
76	54.6\\
77	49.7778\\
78	56.8571\\
79	54.875\\
80	55.7778\\
81	64.1429\\
82	67.2\\
83	57.375\\
84	55.4\\
85	54.5556\\
86	64.875\\
87	52.6\\
88	60.625\\
89	60.125\\
90	60.1\\
91	54\\
92	60.75\\
93	60\\
94	60.7778\\
95	59.4286\\
96	55.8\\
97	55.5556\\
98	53.1429\\
99	52\\
};\label{fig:parameter:mob:2}
\addlegendentry{$b=4\cdot 10^{-5}$};

\addplot [color=green,line width=1.5pt,dashed]
  table[row sep=crcr]{0	57\\
1	53.1667\\
2	48.6\\
3	47.5\\
4	47.25\\
5	50.75\\
6	49\\
7	51\\
8	51.75\\
9	47.25\\
10	47.25\\
11	45.3333\\
12	47.6667\\
13	44\\
14	47\\
15	46.25\\
16	46\\
17	49.25\\
18	48.4\\
19	50\\
20	48.2\\
21	42\\
22	46\\
23	49.125\\
24	50.2857\\
25	50.125\\
26	48.8889\\
27	50.2857\\
28	49.1111\\
29	48.625\\
30	51.5\\
31	52.6667\\
32	48.8889\\
33	47.1429\\
34	48\\
35	48.1111\\
36	49.25\\
37	46\\
38	45.4\\
39	52.75\\
40	48.375\\
41	50.1667\\
42	46.25\\
43	46.6667\\
44	50.125\\
45	48.5\\
46	47\\
47	54.25\\
48	48.7778\\
49	53.5556\\
50	49.625\\
51	51.4\\
52	48.625\\
53	43.5\\
54	48.25\\
55	47.5\\
56	47.6667\\
57	49.4444\\
58	50.25\\
59	50.8\\
60	46\\
61	49\\
62	47.6667\\
63	55.2\\
64	48.75\\
65	48\\
66	47.6\\
67	49.25\\
68	48.8333\\
69	49.625\\
70	48.8571\\
71	51\\
72	49.625\\
73	50.5\\
74	47.25\\
75	49.3333\\
76	51.875\\
77	49.8\\
78	49.2222\\
79	49.7143\\
80	49.25\\
81	53.8\\
82	48.8889\\
83	52.7778\\
84	50.4\\
85	52.8\\
86	50.6667\\
87	45.4286\\
88	50.2\\
89	49.875\\
90	52.3333\\
91	52.8571\\
92	47.6667\\
93	53.75\\
94	49.3333\\
95	51.2222\\
96	51.5556\\
97	48.8571\\
98	47.6\\
99	54.6667\\
};\label{fig:parameter:mob:3}
\addlegendentry{$b=1\cdot 10^{-4}$};

\addplot [color=cyan,line width=1.5pt,solid]
  table[row sep=crcr]{0	56.5\\
1	49.6667\\
2	47\\
3	46.5\\
4	47.6\\
5	47.25\\
6	45.75\\
7	45.3333\\
8	45.6667\\
9	43.3333\\
10	43\\
11	45.6667\\
12	47\\
13	50\\
14	47.6667\\
15	46\\
16	44\\
17	44.25\\
18	44.2\\
19	46.5\\
20	44\\
21	46.75\\
22	45.8\\
23	43.25\\
24	47.1111\\
25	49.25\\
26	47\\
27	50.7778\\
28	46.7143\\
29	44.5\\
30	44.6667\\
31	43\\
32	45.2\\
33	43\\
34	50\\
35	46.5556\\
36	47.875\\
37	48.3\\
38	47.5714\\
39	48\\
40	44.25\\
41	47.8\\
42	46.2\\
43	45.375\\
44	45.4\\
45	45.2\\
46	46\\
47	47.4444\\
48	49.7\\
49	46.0909\\
50	49.8889\\
51	49\\
52	48.8333\\
53	44.6\\
54	35.8\\
55	46\\
56	49.3636\\
57	47\\
58	49.7778\\
59	48.4444\\
60	48.5714\\
61	43\\
62	50.4\\
63	48.875\\
64	49.4\\
65	50.9\\
66	46.75\\
67	45.6\\
68	49.375\\
69	50.375\\
70	47.1111\\
71	48.875\\
72	48.5556\\
73	46.6364\\
74	47.5556\\
75	46.4286\\
76	42.1667\\
77	48.4444\\
78	44.5\\
79	47.3333\\
80	46.8571\\
81	46.5\\
82	51.5\\
83	48\\
84	47.5\\
85	49.2\\
86	48\\
87	48.3333\\
88	46.4444\\
89	50.3636\\
90	49.4444\\
91	46.4286\\
92	45.6667\\
93	45\\
94	48\\
95	49\\
96	48.7\\
97	48.8889\\
98	48.5556\\
99	54.1429\\
};\label{fig:parameter:mob:4}
\addlegendentry{$b=3\cdot 10^{-4}$};

\end{axis}
\end{tikzpicture}%

%% file: vary_penalty.tikz
% This file was created by matlab2tikz v0.4.7 running on MATLAB 9.0.
% Copyright (c) 2008--2014, Nico Schlömer <nico.schloemer@gmail.com>
% All rights reserved.
% Minimal pgfplots version: 1.3
% 
% The latest updates can be retrieved from
%   http://www.mathworks.com/matlabcentral/fileexchange/22022-matlab2tikz
% where you can also make suggestions and rate matlab2tikz.
% 
\begin{tikzpicture}

\begin{axis}[%
width=1.5\figurewidth,
height=0.95\figureheight,
at={(0\figurewidth,0\figureheight)},
scale only axis,
separate axis lines,
every outer x axis line/.append style={black},
every x tick label/.append style={font=\color{black}},
xmin=0,
xmax=99,
major tick length=0.15cm,
minor tick length=0.075cm,
tick style={thick,color=black},
xtick={0,50,99},
every outer y axis line/.append style={black},
every y tick label/.append style={font=\color{black}},
ymin=45,
ymax=80,
ylabel style={align=center},
axis background/.style={fill=white},
legend style={legend cell align=left,align=left,draw=black},
legend pos=outer north east
]
\addplot [color=red,line width=1.5pt,dotted]
  table[row sep=crcr]{0	72.4\\
1	63\\
2	59.75\\
3	63.3333\\
4	70.3333\\
5	79.3333\\
6	70.6667\\
7	68\\
8	68\\
9	66.5\\
10	58\\
11	59\\
12	67\\
13	67.5\\
14	59.5\\
15	62.5\\
16	58\\
17	76.3333\\
18	74.6667\\
19	79.6667\\
20	79\\
21	68.3333\\
22	65.6667\\
23	74.3333\\
24	70.6667\\
25	56\\
26	52.5\\
27	71.6667\\
28	72.3333\\
29	67\\
30	59.6667\\
31	66.4\\
32	63.3333\\
33	58.3333\\
34	65\\
35	66.3333\\
36	61.3333\\
37	57.3333\\
38	52.3333\\
39	56\\
40	63.3333\\
41	58\\
42	66\\
43	51.3333\\
44	63.6667\\
45	60.3333\\
46	56.3333\\
47	68.6667\\
48	63.6667\\
49	65\\
50	70\\
51	57.6667\\
52	68.8\\
53	50.6667\\
54	62\\
55	62.6667\\
56	64.3333\\
57	56.6667\\
58	72\\
59	67\\
60	68.3333\\
61	65.6\\
62	55\\
63	54\\
64	61\\
65	66.6667\\
66	68.6667\\
67	58\\
68	62.2\\
69	59\\
70	69.3333\\
71	64.75\\
72	56.6667\\
73	53\\
74	54.6667\\
75	53.3333\\
76	61.3333\\
77	63.3333\\
78	58.3333\\
79	68\\
80	61\\
81	73.4\\
82	76.3333\\
83	56.3333\\
84	58.4\\
85	70.75\\
86	64\\
87	63.75\\
88	69.5\\
89	63.5\\
90	59.75\\
91	64.6\\
92	70.75\\
93	73.5\\
94	64.2\\
95	68.3333\\
96	73\\
97	70\\
98	58.6667\\
99	69.75\\
};\label{fig:parameter:penalty:1}
\addlegendentry{$s=10^{4}$};

\addplot [color=blue,line width=1.5pt,dashdotted]
  table[row sep=crcr]{0	64.25\\
1	58.5556\\
2	58.5714\\
3	60\\
4	57.6667\\
5	59.5\\
6	62.4\\
7	53.5\\
8	52.25\\
9	57.6667\\
10	59\\
11	61\\
12	72.3333\\
13	61.75\\
14	65.8\\
15	51\\
16	56.3333\\
17	58.75\\
18	57.25\\
19	57.25\\
20	52\\
21	48.5\\
22	52.4\\
23	54.25\\
24	50.4\\
25	58\\
26	58.6667\\
27	57\\
28	51.5\\
29	58.5\\
30	55.4286\\
31	53.5\\
32	50.5\\
33	55.75\\
34	58.8\\
35	56.5\\
36	51.2222\\
37	59.5556\\
38	57.25\\
39	64.75\\
40	46.2222\\
41	50.875\\
42	53\\
43	60.5\\
44	57.6\\
45	59.3333\\
46	60.4286\\
47	61.625\\
48	56.5\\
49	56.6\\
50	53.6\\
51	51.4444\\
52	54.5714\\
53	51.8333\\
54	59.3333\\
55	62.4444\\
56	61\\
57	52\\
58	50\\
59	50.4\\
60	63.75\\
61	60.8\\
62	60.8333\\
63	56.5\\
64	55\\
65	55.7143\\
66	55.6\\
67	54.3333\\
68	55.1667\\
69	48.5\\
70	47.4\\
71	50.2\\
72	51.625\\
73	53.8333\\
74	52.4444\\
75	59.3\\
76	54.6\\
77	49.7778\\
78	56.8571\\
79	54.875\\
80	55.7778\\
81	64.1429\\
82	67.2\\
83	57.375\\
84	55.4\\
85	54.5556\\
86	64.875\\
87	52.6\\
88	60.625\\
89	60.125\\
90	60.1\\
91	54\\
92	60.75\\
93	60\\
94	60.7778\\
95	59.4286\\
96	55.8\\
97	55.5556\\
98	53.1429\\
99	52\\
};\label{fig:parameter:penalty:2}
\addlegendentry{$s=10^{6}$};

\addplot [color=green,line width=1.5pt,dashed]
  table[row sep=crcr]{0	70\\
1	56.25\\
2	55.25\\
3	64.3333\\
4	56\\
5	59.8\\
6	46.2\\
7	51.5714\\
8	54.25\\
9	63.3333\\
10	67\\
11	74\\
12	57\\
13	51\\
14	54.8\\
15	54\\
16	63.75\\
17	51.8333\\
18	55.1667\\
19	50.9091\\
20	53.8\\
21	52\\
22	53\\
23	49.375\\
24	53.4545\\
25	49.375\\
26	51.5\\
27	62\\
28	50\\
29	54.9091\\
30	58.25\\
31	56.25\\
32	47.5\\
33	59.25\\
34	54\\
35	47.5\\
36	50.1667\\
37	59\\
38	56.5\\
39	51.25\\
40	46.1429\\
41	63.1111\\
42	53.3333\\
43	61.25\\
44	55.6667\\
45	58.3636\\
46	54.5\\
47	62.7\\
48	55.875\\
49	52.5\\
50	50.75\\
51	58.5\\
52	58.75\\
53	54.2308\\
54	64.7\\
55	70.1111\\
56	61.2857\\
57	59.8571\\
58	46.25\\
59	69.25\\
60	54.8889\\
61	60.0909\\
62	57.75\\
63	60.3333\\
64	62.3333\\
65	68.7778\\
66	58.25\\
67	59.9091\\
68	65.8333\\
69	51\\
70	55\\
71	47.8\\
72	60.9\\
73	49.6667\\
74	55.4167\\
75	60.3333\\
76	57\\
77	59.8182\\
78	59.1429\\
79	50.25\\
80	57.9091\\
81	53.8571\\
82	66.4\\
83	63.7647\\
84	56.375\\
85	63.9167\\
86	56.4444\\
87	55\\
88	55.75\\
89	56.5385\\
90	61.8\\
91	59.8\\
92	57.9167\\
93	67.2222\\
94	63.4\\
95	55.8333\\
96	54.8333\\
97	58\\
98	54.125\\
99	58.7778\\
};\label{fig:parameter:penalty:3}
\addlegendentry{$s=10^{8}$};

\addplot [color=cyan,line width=1.5pt,solid]
  table[row sep=crcr]{0	77.4\\
1	62.3333\\
2	54.625\\
3	52.3333\\
4	53.75\\
5	53.4\\
6	46\\
7	51.1111\\
8	49.5\\
9	54\\
10	62.3333\\
11	66.3333\\
12	56\\
13	56.75\\
14	46.6\\
15	51.8\\
16	71\\
17	55.5714\\
18	53.8\\
19	65.3333\\
20	45.8\\
21	51.8\\
22	52.5\\
23	53\\
24	55.8182\\
25	53.375\\
26	61.5\\
27	47.6667\\
28	47.4\\
29	57.4545\\
30	55.9167\\
31	51.125\\
32	49.1667\\
33	62\\
34	54.8\\
35	51.8\\
36	54.5455\\
37	61.8182\\
38	58.625\\
39	56.5\\
40	54.2\\
41	61.7\\
42	47.8571\\
43	57.25\\
44	66.6667\\
45	59.6667\\
46	61.375\\
47	58.9\\
48	60.25\\
49	46.6667\\
50	54.8\\
51	57.8462\\
52	64\\
53	53\\
54	65.5\\
55	64.8889\\
56	61.7143\\
57	55.8571\\
58	54.8\\
59	71.75\\
60	54.5556\\
61	64.6364\\
62	55.375\\
63	56.5\\
64	57.5556\\
65	61.4444\\
66	58.1429\\
67	58\\
68	58.1667\\
69	54\\
70	49.1667\\
71	51\\
72	53.3\\
73	49.3333\\
74	55.3077\\
75	61.0833\\
76	63.6\\
77	57.6364\\
78	58.2857\\
79	51.8889\\
80	55.6667\\
81	57.1429\\
82	65.2\\
83	63.5652\\
84	54.75\\
85	57.5556\\
86	63.3846\\
87	58.0769\\
88	49.5\\
89	58.6429\\
90	67.2727\\
91	59.7\\
92	61\\
93	61.1111\\
94	57.2\\
95	63\\
96	51.6667\\
97	60.9091\\
98	54.125\\
99	58.8889\\
};\label{fig:parameter:penalty:4}
\addlegendentry{$s=10^{9}$};

\end{axis}
\end{tikzpicture}%

%% file: iter_Bench2.tikz
% This file was created by matlab2tikz v0.4.7 running on MATLAB 9.0.
% Copyright (c) 2008--2014, Nico Schlömer <nico.schloemer@gmail.com>
% All rights reserved.
% Minimal pgfplots version: 1.3
% 
% The latest updates can be retrieved from
%   http://www.mathworks.com/matlabcentral/fileexchange/22022-matlab2tikz
% where you can also make suggestions and rate matlab2tikz.
% 
\begin{tikzpicture}

\begin{axis}[%
width=1.5\figurewidth,
height=0.95\figureheight,
at={(0\figurewidth,0\figureheight)},
scale only axis,
separate axis lines,
every outer x axis line/.append style={black},
every x tick label/.append style={font=\color{black}},
xmin=0,
xmax=279,
major tick length=0.15cm,
minor tick length=0.075cm,
tick style={thick,color=black},
xtick={0,140,279},
every outer y axis line/.append style={black},
every y tick label/.append style={font=\color{black}},
ymin=80,
ymax=200,
ylabel style={align=center},
axis background/.style={fill=white},
legend style={legend cell align=left,align=left,draw=black},
legend pos=outer north east
]
\addplot [color=blue,line width=1.5pt,solid,forget plot]
  table[row sep=crcr]{0	184.4\\
1	116.667\\
2	108.5\\
3	98\\
4	105\\
5	112.5\\
6	112\\
7	119\\
8	125\\
9	117.5\\
10	106\\
11	105.5\\
12	104\\
13	109.5\\
14	111\\
15	105.5\\
16	107.5\\
17	107.5\\
18	105.5\\
19	102\\
20	101\\
21	97.5\\
22	94.5\\
23	101\\
24	97.5\\
25	100\\
26	116.5\\
27	114\\
28	108\\
29	85\\
30	101\\
31	96.6667\\
32	103.667\\
33	109.333\\
34	97.5\\
35	97.5\\
36	109.667\\
37	109.667\\
38	103\\
39	105\\
40	105.333\\
41	112.333\\
42	117.667\\
43	108.667\\
44	121\\
45	111\\
46	115.333\\
47	113\\
48	108.333\\
49	96.6667\\
50	140.667\\
51	117\\
52	121\\
53	99.5\\
54	108.667\\
55	102\\
56	114.333\\
57	104.667\\
58	99\\
59	124\\
60	119.667\\
61	103\\
62	106.667\\
63	105\\
64	115.667\\
65	104.333\\
66	124.75\\
67	105\\
68	121\\
69	117.667\\
70	133\\
71	128.667\\
72	120.667\\
73	110.25\\
74	103\\
75	102.667\\
76	104\\
77	116.75\\
78	123\\
79	134.25\\
80	115.667\\
81	120.75\\
82	116\\
83	116\\
84	134.75\\
85	136\\
86	102\\
87	111.667\\
88	110.333\\
89	107.333\\
90	105\\
91	101\\
92	118\\
93	112\\
94	126.667\\
95	125.333\\
96	111.5\\
97	115.5\\
98	153.25\\
99	122\\
100	119\\
101	111.25\\
102	114\\
103	135.25\\
104	153.667\\
105	119\\
106	116.25\\
107	132.75\\
108	107\\
109	123.75\\
110	116.667\\
111	99.3333\\
112	113.333\\
113	118\\
114	109.75\\
115	108\\
116	113.75\\
117	155.5\\
118	104.25\\
119	126\\
120	129.5\\
121	127.667\\
122	125.2\\
123	149.6\\
124	119\\
125	118.6\\
126	116.75\\
127	137.5\\
128	127\\
129	108\\
130	134.333\\
131	118\\
132	126.25\\
133	109\\
134	122.8\\
135	147.5\\
136	109.75\\
137	102.25\\
138	157.5\\
139	123\\
140	126.8\\
141	118.5\\
142	142.25\\
143	142.5\\
144	115.667\\
145	119.25\\
146	113.8\\
147	117.8\\
148	140.25\\
149	102\\
150	120\\
151	118\\
152	122.5\\
153	116.4\\
154	111\\
155	126.25\\
156	111.4\\
157	116\\
158	114.5\\
159	130.4\\
160	129.5\\
161	124\\
162	117.8\\
163	132.25\\
164	150.75\\
165	120\\
166	148.5\\
167	151\\
168	117.75\\
169	121.75\\
170	119.2\\
171	110.25\\
172	120.25\\
173	107.8\\
174	109.6\\
175	140.2\\
176	130.5\\
177	117.25\\
178	143\\
179	128\\
180	113.2\\
181	114.4\\
182	151\\
183	129.8\\
184	99.75\\
185	113\\
186	110.5\\
187	110\\
188	159\\
189	121\\
190	124.8\\
191	126.4\\
192	107.8\\
193	124.4\\
194	111.25\\
195	125.75\\
196	106.2\\
197	118.8\\
198	129.75\\
199	134\\
200	125.5\\
201	108.6\\
202	115.75\\
203	112.6\\
204	126.4\\
205	144.4\\
206	104.25\\
207	113.2\\
208	132.2\\
209	119.75\\
210	116.75\\
211	112.6\\
212	135.8\\
213	114\\
214	128.4\\
215	115.8\\
216	105.2\\
217	142.75\\
218	119.75\\
219	133.6\\
220	136.2\\
221	128.2\\
222	108.75\\
223	128.5\\
224	108\\
225	124.25\\
226	107.4\\
227	124.8\\
228	129.6\\
229	114.6\\
230	113\\
231	111.4\\
232	113.6\\
233	121\\
234	114.2\\
235	134.6\\
236	159.2\\
237	110.5\\
238	125.4\\
239	135.75\\
240	119.8\\
241	121.5\\
242	111.6\\
243	120\\
244	144.2\\
245	117.2\\
246	127.8\\
247	112.4\\
248	115.2\\
249	107.4\\
250	128\\
251	126.2\\
252	110\\
253	128.667\\
254	109.4\\
255	144.6\\
256	111.5\\
257	113.5\\
258	121.6\\
259	123.2\\
260	115.4\\
261	108.6\\
262	114.2\\
263	118.6\\
264	117.4\\
265	118.4\\
266	130\\
267	112.25\\
268	140.8\\
269	145.6\\
270	110\\
271	122.6\\
272	114.2\\
273	130.6\\
274	103.2\\
275	112\\
276	115.4\\
277	106.8\\
278	108\\
279	150.5\\
};
\end{axis}
\end{tikzpicture}%